\documentclass[letterpaper,11pt,reqno]{amsart} 
\usepackage[portrait,margin=1in]{geometry}
\usepackage{authblk}

\usepackage{comment}
\usepackage{systeme}
\usepackage{mathrsfs,xfrac} 
\usepackage[colorlinks=true,linkcolor=blue,citecolor=blue,urlcolor=blue]{hyperref} 
\usepackage{amsmath,amssymb,amsthm,amsfonts,amsbsy,latexsym,dsfont,color} 
\usepackage[numeric,initials,nobysame]{amsrefs} 
\usepackage[foot]{amsaddr}
\usepackage{graphicx}
\usepackage{float}
\usepackage{authblk}
\usepackage[utf8]{inputenc}
\usepackage[english]{babel}
\usepackage{comment}

\usepackage[textsize=tiny]{todonotes}
\usepackage{regexpatch}
\makeatletter
\xpatchcmd{\@todo}{\setkeys{todonotes}{#1}}{\setkeys{todonotes}{inline,#1}}{}{}
\makeatother

\usepackage{enumerate}
\newenvironment{enumeratei}{\begin{enumerate}[\upshape i)]}{\end{enumerate}}
\newenvironment{enumeratea}{\begin{enumerate}[\upshape a)]\setlength{\itemsep}{1ex}}{\end{enumerate}}

\usepackage{paralist}

\newtheorem{thm}{Theorem}[section]
\newtheorem{lem}[thm]{Lemma}
\newtheorem{cor}[thm]{Corollary}
\newtheorem{prop}[thm]{Proposition}
\newtheorem{defn}[thm]{Definition}

\theoremstyle{definition}
\newtheorem{rem}[thm]{Remark}

\renewcommand{\le}{\leqslant} 
\renewcommand{\ge}{\geqslant} 
\renewcommand{\leq}{\leqslant} 
\renewcommand{\geq}{\geqslant}

\newcommand{\ind}{\mathds{1}}
\newcommand{\eps}{\varepsilon}

\newcommand{\abs}[1]{\left\vert#1\right\vert}
\newcommand{\set}[1]{\left\{#1\right\}}

\newcommand{\eg}{\emph{e.g.,}}

\newcommand{\probc}{\stackrel{\mathrm{P}}{\longrightarrow}}

\def\qed{ \hfill $\blacksquare$}  
    
     \let\gl=\lambda

\newcommand{\cA}{\mathcal{A}}\newcommand{\cC}{\mathcal{C}}
\newcommand{\cE}{\mathcal{E}}\newcommand{\cF}{\mathcal{F}}
\newcommand{\cH}{\mathcal{H}}\newcommand{\cI}{\mathcal{I}}

\newcommand{\cM}{\mathcal{M}}\newcommand{\cN}{\mathcal{N}}
\newcommand{\cP}{\mathcal{P}}
\newcommand{\cS}{\mathcal{S}}\newcommand{\cT}{\mathcal{T}}

\newcommand{\vP}{\mathbf{P}}\newcommand{\vQ}{\mathbf{Q}}

\newcommand{\vs}{\mathbf{s}}\newcommand{\vt}{\mathbf{t}}

\newcommand{\mvx}{\boldsymbol{x}}\newcommand{\mvy}{\boldsymbol{y}}

\newcommand{\mvvarpi}{\boldsymbol{\varpi}}
   
\newcommand{\fF}{\mathfrak{F}}

\newcommand{\fJ}{\mathfrak{J}}

\newcommand{\fP}{\mathfrak{P}}

\newcommand{\bb}[1]{\mathbb{#1}}

\newcommand{\bL}{\mathbb{L}}
\newcommand{\bN}{\mathbb{N}}
\newcommand{\bR}{\mathbb{R}}
\newcommand{\bT}{\mathbb{T}}

\newcommand{\bZ}{\mathbb{Z}}        

\newcommand{\sE}{\mathscr{E}}

\DeclareMathOperator{\E}{\mathds{E}}
\DeclareMathOperator{\pr}{\mathds{P}}

\DeclareMathOperator{\argmax}{argmax}

\DeclareMathOperator{\N}{N}

\newcommand{\bbT}{\mathbb{T}}
\newcommand{\TT}{\mathcal{T}}
\newcommand{\bs}{\mathbf{s}}

\newcommand{\bt}{\mathbf{t}}
\newcommand{\bfomega}{{\boldsymbol \omega}}
\newcommand{\Zbold}{{\mathbb{Z}}}
\newcommand{\prob}{\mathbb{P}}
\definecolor{aqua}{rgb}{0.0, 1.0, 1.0}
\definecolor{boo}{rgb}{1.0, 0.0, 1.0}

\newcommand{\probfr}{\stackrel{\mbox{$\operatorname{P}$-\bf fr}}{\longrightarrow}}
\newcommand{\probcrf}{\stackrel{\mbox{$\operatorname{P}$-\bf efr}}{\longrightarrow}}

\newcommand{\Efr}{\stackrel{\mbox{$\E$-\bf fr}}{\longrightarrow}}

\newcommand{\convas}{\stackrel{\mathrm{a.s.}}{\longrightarrow}}
\newcommand{\convd}{\stackrel{d}{\longrightarrow}}
\newcommand{\convp}{\stackrel{P}{\longrightarrow}}

\newtheorem{lemma}[thm]{Lemma}

\newcommand{\stod}{\preceq_{\mathrm{st}}}
 \DeclareMathOperator{\BP}{BP}

\newcommand{\Poi}{{\sf Poisson} }

\newcommand{\fringe}{{\sf fringe} }

\usepackage[mathscr]{euscript}
\DeclareMathAlphabet{\mathscrbf}{OMS}{mdugm}{b}{n}

\DeclareFontFamily{U}{BOONDOX-calo}{\skewchar\font=45 }
\DeclareFontShape{U}{BOONDOX-calo}{m}{n}{
  <-> s*[1.05] BOONDOX-r-calo}{}
\DeclareFontShape{U}{BOONDOX-calo}{b}{n}{
  <-> s*[1.05] BOONDOX-b-calo}{}
\DeclareMathAlphabet{\mathcalb}{U}{BOONDOX-calo}{m}{n}
\SetMathAlphabet{\mathcalb}{bold}{U}{BOONDOX-calo}{b}{n}
\DeclareMathAlphabet{\mathbcalb}{U}{BOONDOX-calo}{b}{n}

\newcommand{\fpm}{\mvvarpi}

\usepackage{tikz}
\usetikzlibrary{calc,intersections,through,backgrounds,shapes.geometric}
\usetikzlibrary{graphs}

\usetikzlibrary{calc,intersections,through,backgrounds,shapes.geometric}
\usetikzlibrary{graphs}
\tikzset{every path/.style={line width=.07 cm}}

\usepackage{subcaption}

\definecolor{cser}{RGB}{0,145,135}

\begin{document}

\title[Evolution of recursive trees with limited memory]{Evolution of recursive trees with limited memory}

\author[Angel]{Omer Angel}
\address{Department of Mathematics, University of British Columbia.}
\email{angel@math.ubc.ca}
\author[Bhamidi]{Shankar Bhamidi}
\address{Department of Statistics and Operations Research, University of North Carolina at Chapel Hill.}
\email{bhamidi@email.unc.edu}
\author[Donderwinkel]{Serte Donderwinkel}
\address{Bernoulli Institute of Mathematics, University of Groningen.}
\email{s.a.donderwinkel@rug.nl}
\author[Maitra]{Neeladri Maitra}
\address{Department of Mathematics, University of Illinois at Urbana-Champaign.}
\email{nmaitra@illinois.edu}
\author[Sakanaveeti]{Akshay Sakanaveeti}
\address{Department of Statistics and Operations Research, University of North Carolina at Chapel Hill.}
\email{sakshay@unc.edu}

\date{}

\begin{abstract}
	Motivated by questions in social networks, distributed computing and probabilistic combinatorics, the last few years have seen increasing interest in network evolution models where new vertices entering the system need to make decisions based on a partial snapshot of the current state of the network. This paper considers a specific variant of the classical random recursive tree dynamics,  where a vertex at time $n+1$ has information only on those vertices that have arrived in the interval $[j(n), n]$ for a sequence $j(n) \uparrow \infty$, and connects to vertices uniformly at random amongst this set. \textcolor{black}{We consider} two different regimes on the density information, termed macroscopic and mesoscopic regimes, \textcolor{black}{which respectively correspond to $j(n)=\theta n$ for some $\theta \in (0,1)$, and $j(n)=n-n^{\beta}$ for some $\beta \in (0,1)$.} \textcolor{black}{Our main interest is in studying} asymptotics of various functionals of the network, including local weak limits and the height of the network. \textcolor{black}{We show that in the macroscopic regime, the local limit is expressed in terms of an associated continuous time branching process that depends on the parameter $\theta$, while it is a $\mathrm{Poisson}(1)$-branching process in the mesoscopic regime for any $\beta \in (0,1)$. Furthermore, the height of the macroscopic tree is logarithmic, which we prove exploiting a connection with scaled-attachment random recursive trees (SARRTs) as studied in \cite{devroye2012depth}, while it is polynomial in the mesoscopic regime; our argument in this latter case relies on a differential equation approach to track the ancestor indices of late-coming vertices, together with a multiscale analysis. {Further, we develop an exploration algorithm to simultaneously reveal the ancestral path of \emph{youngest} vertices. Using this algorithm, we show that in the mesoscopic regime, the global structure experiences a phase transition at $\beta=1/2$.}} 	 
\end{abstract}

	\maketitle


    \noindent  
       

    \begin{figure}[h]
        \centering
        \includegraphics[scale=.6]{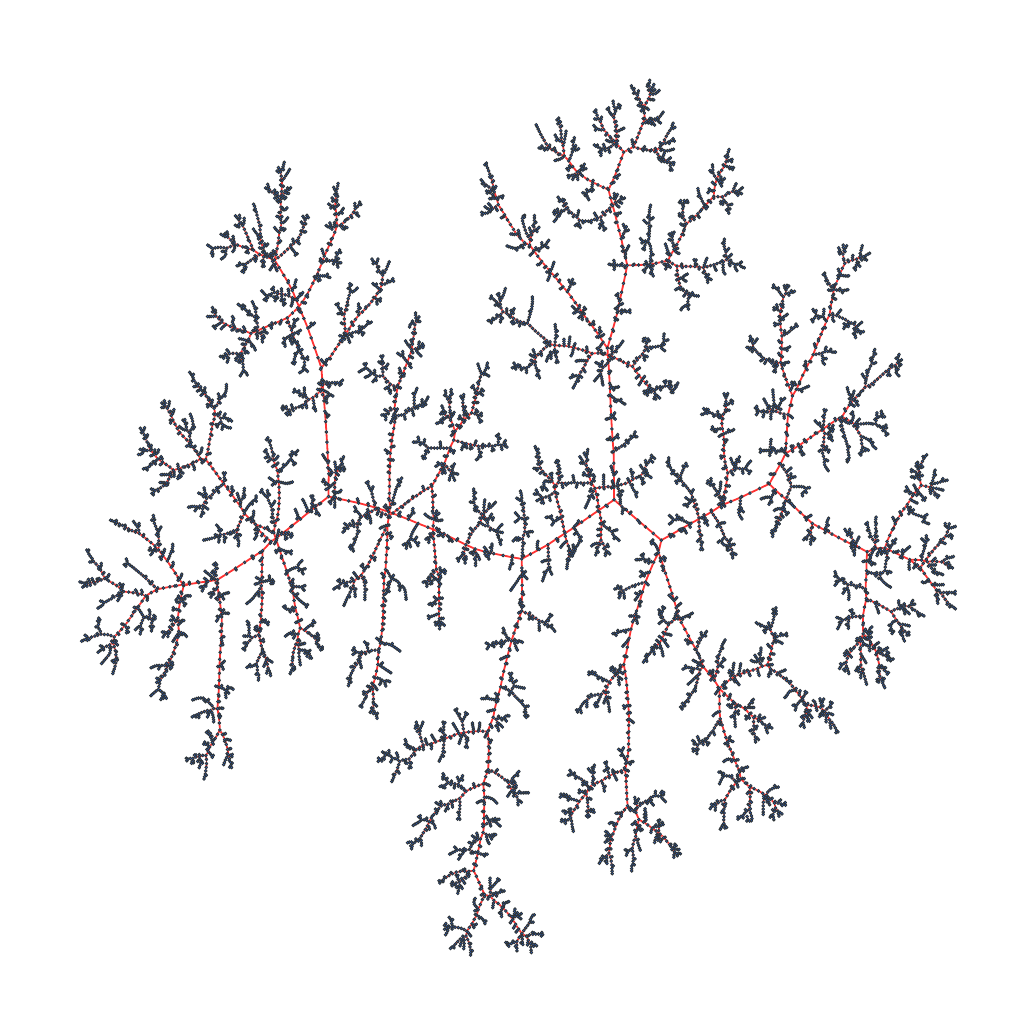} 
        \caption{A recursive tree with limited memory in the mesoscopic regime with 10,000 nodes and with $\beta=0.75$ (see Section \ref{sec:model} for relevant definitions). Section \ref{sec:sims} has more pictures.}
        \label{fig:your_image} 
    \end{figure}
        
	\section{Introduction}

Driven by questions in probabilistic combinatorics, as well as the vast quantities of network valued data in the last few years, the last few years have seen the development of a number of new mathematical network models to understand the emergence of macroscopic properties in such systems~\cites{albert2002statistical, newman2003structure,newman2010networks, bollobas2001random,durrett-rg-book,van2009random}. One sub-area in this rapidly burgeoning field is dynamic or temporal networks,~\cites{holme2012temporal, masuda2016guidance}.

The preferential attachment class of network evolution models assumes new vertices enter the system and then make probabilistic choices for connecting to existing vertices based on the {\bf entire} state of the current network.  There has been increasing interest in understanding the impact of limited information availability for new vertices on the system's state and its role in the evolution of real world systems. This has motivated (at least) two distinct lines of recent research:
\begin{enumeratea}
\item  Network evolution schemes where individuals try to form connections via limited scope explorations of the underlying network,  such as random walks,  initialized at uniform locations in the network run for a few number of steps~\cites{krapvisky2023magic, banerjee2024co, berry2024random,englander2025structural}. 

\item  More relevant for this work, incorporating limited temporal information,  where new individuals have information only on a subset of individuals in the network and need to make their (probabilistic) decisions based on this information ~\cites{baccelli2019renewal,king2021fluid,dey2022asymptotic,BBDS04_meso,BBDS04_macro}. 
\end{enumeratea}
The goal of this paper is to propose and understand a specific class of such models, where new individuals only have information on more recent individuals in the network. \textcolor{black}{We are motivated by real-world network scenarios where older individuals, even though had been influential during some time period, lose relevance over time, and thus are less preferred by younger incoming individuals for developing connections. Consider for example, networks of research papers where a pair of research papers are linked if one refers the another. A new incoming research paper is more likely to refer to a recent, well-explained survey, than a paper that is quite old by now, but greatly influential in its own time, perhaps simply because it was always a little harder to read. Let us proceed with the formal model definition next.}


\subsection{Network model}
\label{sec:model}
We consider a model of recursively growing random trees $(\mathcal{T}_n)_{n \geq 1}$ as follows. $\mathcal{T}_1$ is the singleton tree with only vertex $1$ and no edges. For $n\geq 1$, given $\mathcal{T}_n$, we construct $\mathcal{T}_{n+1}$ by -- 
\begin{enumeratei}
    \item letting the vertex set of $\mathcal{T}_{n+1}$ to be the union of the vertex set of $\mathcal{T}_n$ and a new vertex $n+1$;
    \item drawing an edge between the vertex $n+1$ and a random vertex of $\mathcal{T}_n$ sampled with probability proportional to $f(D_n(v))\mathbf{1}_{\{j(n)\leq v \leq n\}}$,
\end{enumeratei}
 where $f:\mathbb{N} \to (0,\infty)$ is an attachment function measuring the attractiveness of individuals in the current network based on their degree, $D_n(v)$ is the degree of the vertex $v$ in the tree $\mathcal{T}_n$, and $j:\mathbb{N}\to \mathbb{N}$ is a monotone {non-decreasing} function satisfying $j(n)\leq n$ for all $n \geq 1$.

 In words, the new vertex $n+1$ arrives in the already constructed network $\mathcal{T}_n$ and has information only on the more recent individuals in the sense that it can connect only to  a \emph{young} vertex which arrived at time $j(n)$ or later; in principle it can use the degree of this vertex to make its decision. \textcolor{black}{One should think of $j(n)$ as a `age-preference threshold' for the vertex $n+1$, beyond which it chooses not to develop connections.} 
 In order to develop the salient features of such models, for this paper, we \textcolor{black}{focus on the case} $f(x)=1$ (uniform attachment) and consider the following two regimes for $j(n)$:
\begin{enumeratea}
    \item {\bf Macroscopic regime:} \textcolor{black}{$j(n)=\lfloor \theta n\rfloor $} 
    The term ``macroscopic'' is meant to evoke the mental picture that in the large network limit, new incoming vertices have a positive density of information, in terms of the network size,  to make their choices. In this regime, the model is a specific example of the so-called scaled attachment recursive trees (SARRT) introduced in \cite{devroye2012depth} and whose depth and height asymptotics were studied in the above paper. 
    \item {\bf Mesoscopic regime:} \textcolor{black}{$j(n)=n-\lfloor n^{\beta}\rfloor$} for some $\beta \in (0,1)$.
\end{enumeratea}
The main goal of this paper is proposing the general model above and describing asymptotics in the specific ``uniform attachment with bounded information'' setting. 

\begin{rem}\label{rem:future_work_general_j_nontree}
    \textcolor{black}{One can also consider the case $j(n)=\lfloor n^{\beta} \rfloor$, in which case we expect the model to mostly display similarities with the $j(n)=1$ case, although we believe there are some potential differences. As the latter case has been extensively pursued in the literature, we don't put any focus on this case in this paper, and leave this exploration for future work. Similarly, in the model definition, one can consider the case when the incoming vertex comes in with more than one edge, in which case the obtained networks are not trees anymore, and are more complex. We leave this for future work as well.}
\end{rem}

\subsection{Simulations}\label{sec:sims}

We present simulations in this section. {In the mesoscopic regime, one can observe the phase transition in the global structure at $\beta=1/2$ implied by Theorems~\ref{thm:ghp_meso},~\ref{thm:meso_star} and~\ref{thm:meso_not_crt}.} 
In the macroscopic regime, there does not seem to be an abrupt change in behavior in terms of $\theta$ based on the simulations, but this requires further investigation. 
\begin{figure}[H] \centering \begin{subfigure}[b]{0.30\textwidth} \centering \includegraphics[width=4.5cm]{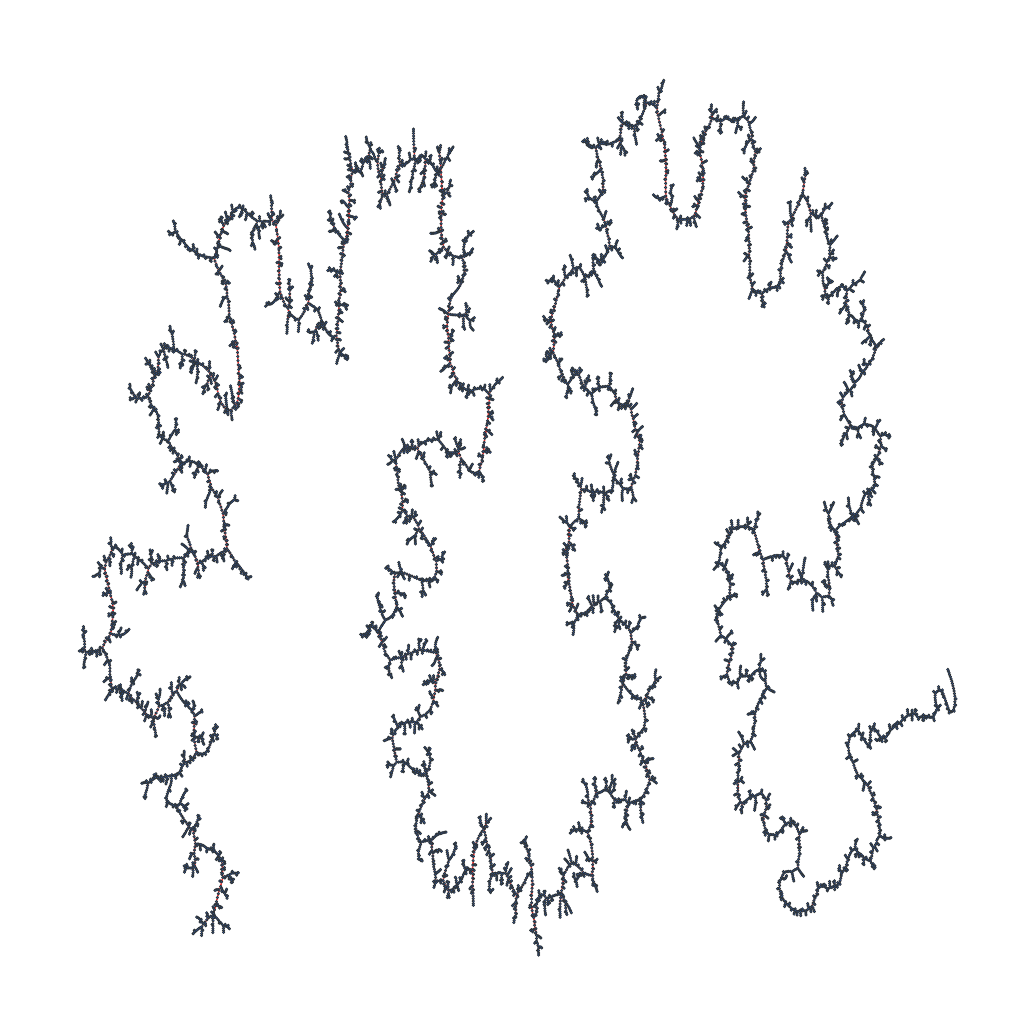} \caption{Mesoscopic regime with 100,000 nodes, $\beta=0.25$} \label{fig:sim_meso_100k_0.25} \end{subfigure} \hfill \begin{subfigure}[b]{0.30\textwidth} \centering \includegraphics[width=4.5cm]{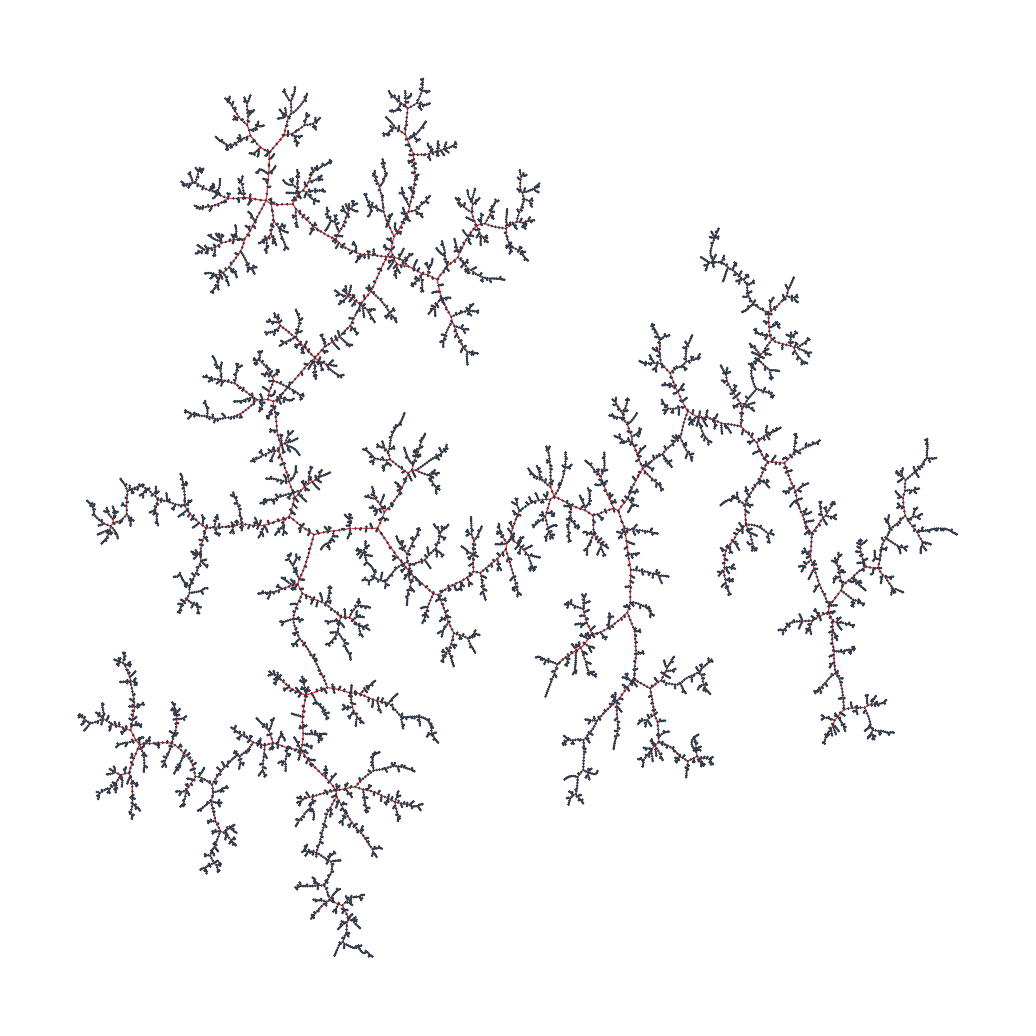} \caption{Mesoscopic regime with 100,000 nodes, $\beta=0.5$} \label{fig:sim_meso_100k_0.5} \end{subfigure} \hfill \begin{subfigure}[b]{0.30\textwidth} \centering \includegraphics[width=4.5cm]{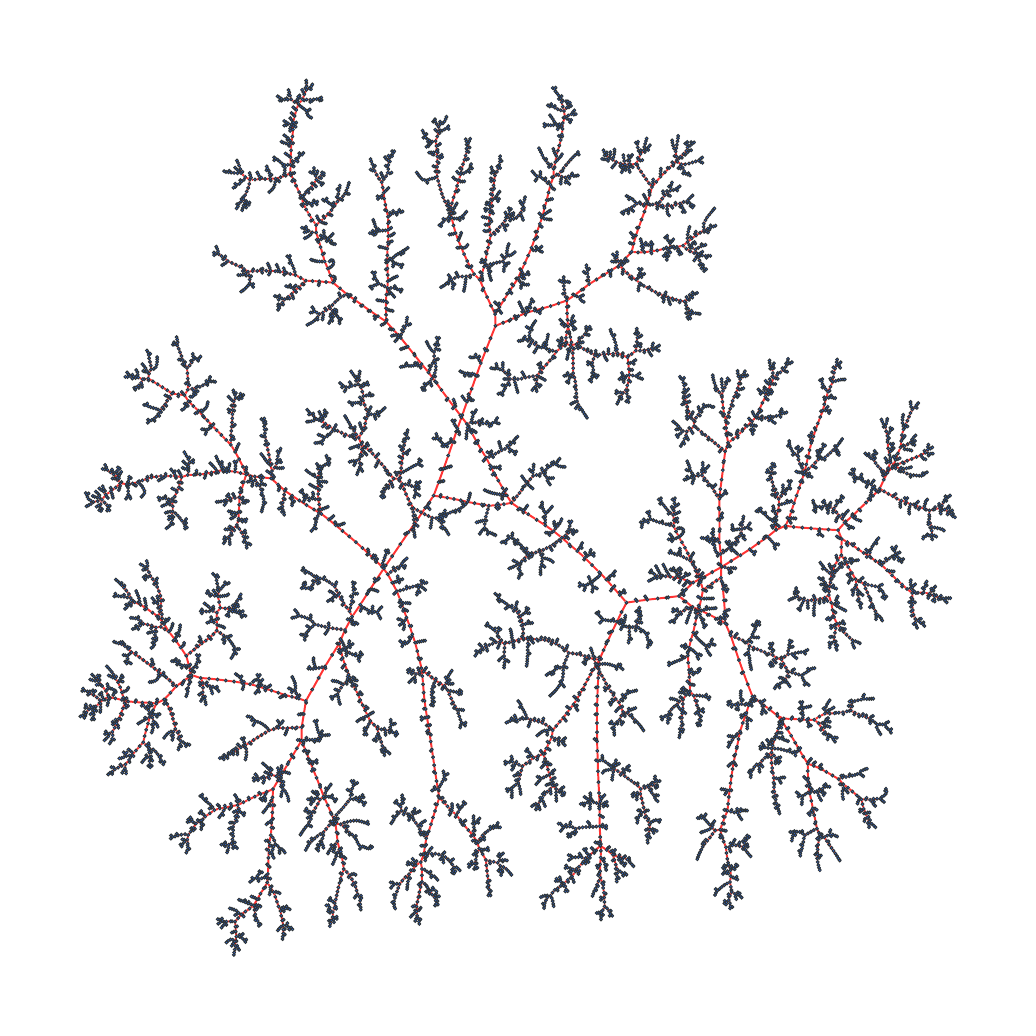} \caption{Mesoscopic regime with 100,000 nodes, $\beta=0.75$} \label{fig:sim_meso_100k_0.75} \end{subfigure} \caption{Pictures of the tree $\cT_n$ in different mesoscopic regimes. As $\beta$ increases and passes through $0.5$, the tree from being more `line-like', becomes more spread out and `fatter', see Theorems \ref{thm:ghp_meso} and \ref{thm:meso_star}.} \label{fig:meso sims} \end{figure} \begin{figure}[H] \centering \begin{subfigure}[b]{0.30\textwidth} \centering \includegraphics[width=4.5cm]{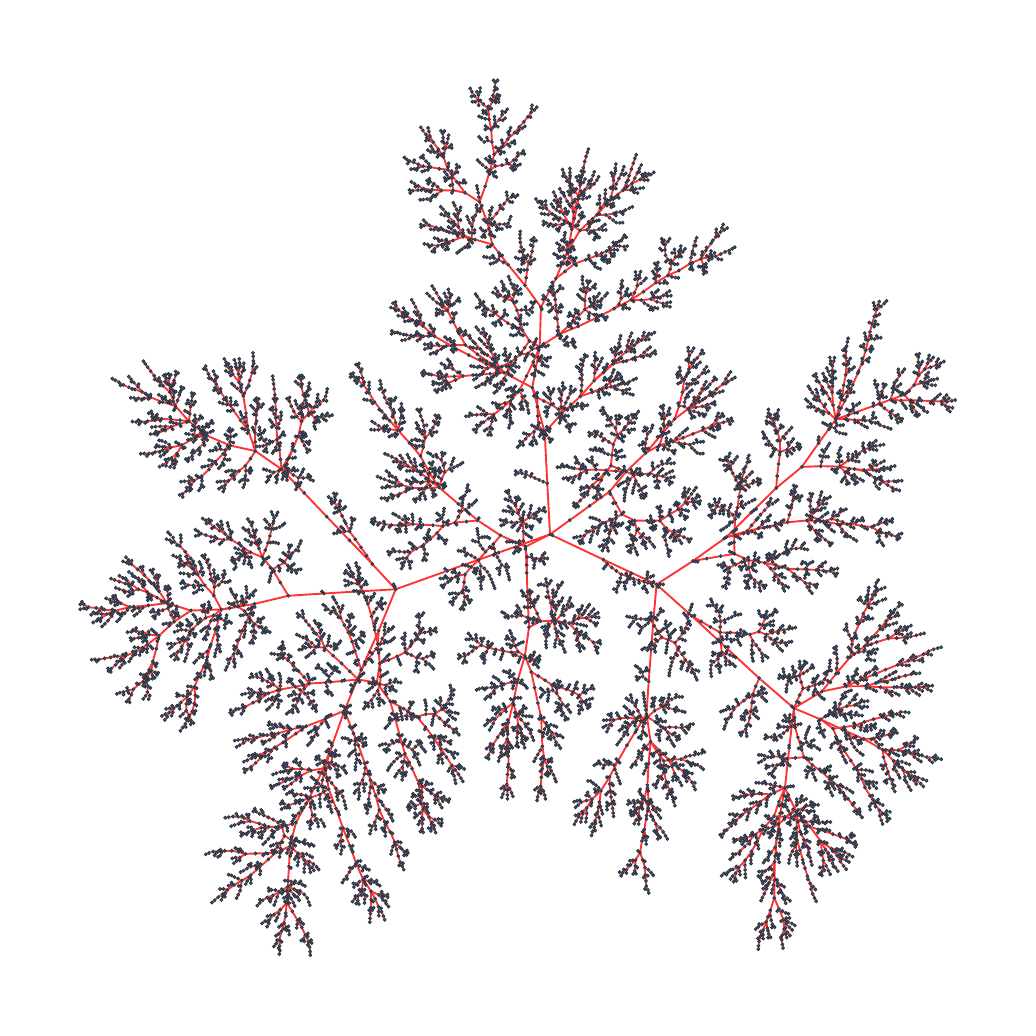} \caption{Macroscopic regime with 100,000 nodes, $\theta=0.25$} \label{fig:sim_macro_100k_0.25} \end{subfigure} \hfill \begin{subfigure}[b]{0.30\textwidth} \centering \includegraphics[width=4.5cm]{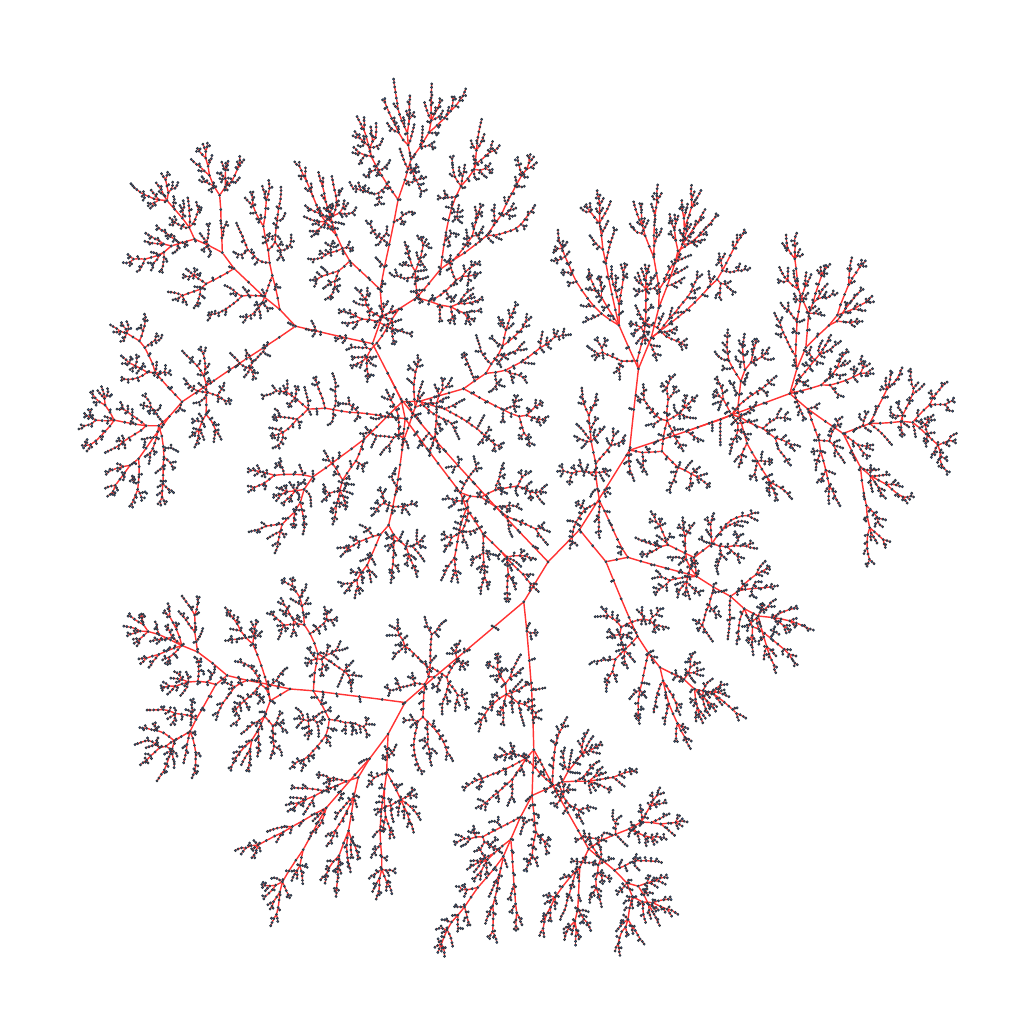} \caption{Macroscopic regime with 100,000 nodes, $\theta=0.5$} \label{fig:sim_macro_100k_0.5} \end{subfigure} \hfill \begin{subfigure}[b]{0.30\textwidth} \centering \includegraphics[width=4.5cm]{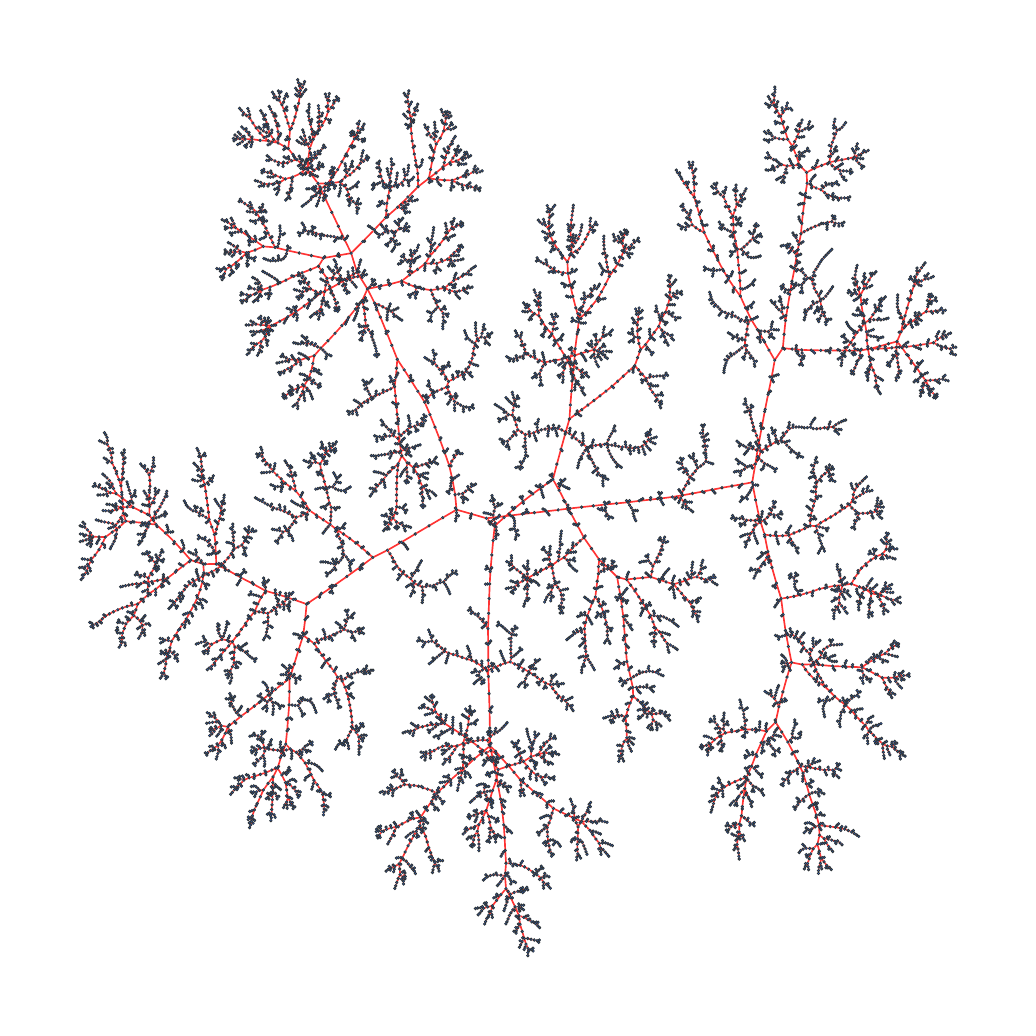} \caption{Macroscopic regime with 100,000 nodes, $\theta=0.75$} \label{fig:sim_macro_100k_0.75} \end{subfigure} \caption{Pictures of the tree $\cT_n$ in different macroscopic regimes.} \label{fig:macro sims} \end{figure}

\medskip

\noindent\textcolor{black}{\textbf{Mathematical notations.} We use $\stod$ for stochastic domination between two real-valued probability measures. For $J\geq 1$ let $[J]:= \set{1,2,\ldots, J}$. If $Y$ has an exponential distribution with rate $\gl$, write this as $Y\sim \exp(\gl)$. Write $\bZ$ for the set of integers, $\bR$ for the real line, $\bN$ for the natural numbers and let $\bR_+:=(0,\infty)$. Write $\convas,\convp,\convd$ for convergence almost everywhere, in probability and in distribution, respectively. For a non-negative function $n\mapsto g(n)$,
we write $f(n)=O(g(n))$ when $|f(n)|/g(n)$ is uniformly bounded, and
$f(n)=o(g(n))$ when $\lim_{n\rightarrow \infty} f(n)/g(n)=0$.
Furthermore, write $f(n)=\Theta(g(n))$ if $f(n)=O(g(n))$ and $g(n)=O(f(n))$.
We write that a sequence of events $(A_n)_{n\geq 1}$
occurs \emph{with high probability} (whp) when $\pr(A_n)\rightarrow 1$ as $n \rightarrow \infty$.}

\medskip

\noindent\textbf{Organization of the paper.} Section~\ref{sec:main-res} contains statements of the main results and a brief overview of related work. For readers unfamiliar with local weak convergence and fringe convergence, precise definitions are given in Section~\ref{sec:local-wll}. Section~\ref{sec:proofs} contains proofs of the main results.

\section{Results}
\label{sec:main-res}

\subsection{\textcolor{black}{Results on} local functionals}
We first start by describing the local limits in the two regimes. 
\subsubsection{\bf Macroscopic regime}\label{subsec:results-macro}

To state our main result, we need some final notation. For $n\geq 1$, let $ \bT_{n} $ be the space of all rooted trees on  $n$ {vertices}. Let $ \bbT =
\cup_{n=0}^\infty \bT_{n} $ be the space of all finite rooted trees.  Here $\bT_{0} = \emptyset $ will be used to represent the empty tree (tree on zero vertices). For any $\bt \in \bbT$, let $\rho_{\bt}$ denote the root of this tree.

\textcolor{black}{Recall $j(n)=\theta n$ for the macroscopic regime where} the parameter $\theta \in (0,1)$. Define $c_\theta  = \log{\theta^{-1}} \in (0,\infty)$. Let $\cN_\theta$ denote a rate $(1-\theta)^{-1}$ Poisson process on the interval $[0,c_\theta]$. Let $\BP_\theta$ denote a continuous time branching process with offspring point process $\cN_\theta$, started with one individual at time $t=0$. \textcolor{black}{Observe that since $c_{\theta}(1-\theta)^{-1}>1$, the branching process survives with positive probability.} Let $T_1$ \textcolor{black}{be an exponentially distributed random variable with mean $1$, which we write as} $T_1\sim \exp(1)$, independent of $\BP_\theta$. Let $\fpm_{\theta}$ denote the distribution of $\BP_{\theta}(T_1)$, viewed as a random finite rooted tree on $\bbT$, where we retain only genealogical information between individuals in $\BP_{\theta}(T_1)$. 

\begin{thm}[Local weak convergence in the macroscopic regime]
\label{thm:lwc-macro}
    Fix $\theta \in (0,1)$. Then $\set{\cT(n):n\geq 1}$ converges in probability in the extended fringe sense (Def.~\ref{def:local-weak}~\eqref{it:fringe-b}) to the unique infinite {\tt sin}-tree with fringe distribution $\fpm_{\theta}$. 
\end{thm}

This result has the following implication. Here,  with some abuse of notation, for any $t\geq 0$,  we use $\cN_\theta(t)$ for the number of points in the point process $\cN_\theta$ and assume $\cN_\theta$ independent of $T_1\sim \exp(1)$ as before. \textcolor{black}{Denote by $N_k(n)$ the number of vertices with degree equal to $k$ in the tree $\cT_n$.}

\begin{cor}
    In the setting of Theorem \ref{thm:lwc-macro}, the degree counts $\set{N_k(n):k\geq 1}$ of the random tree $\cT_n$  as in~\eqref{eqn:deg-count}  satisfy for $k\geq 1$,
\[
\frac{N_k(n)}{n} \convp \prob(\cN_\theta(T_1) =k-1) \text{ as } n\to\infty.
\]
\end{cor}

\subsubsection{\bf Mesoscopic regime}

Let $\fpm_{\Poi, 1}$ denote the distribution of a critical Galton-Watson branching process with Poisson mean one offspring on $\bbT$. Recall from \cite{grimmett1980random,aldous-fringe} that if $\cT_n$ is random rooted tree sampled uniformly amongst all $n^{n-1}$ trees on $[n]$ labelled vertices then $\cT_n$ converges in the probabilistic fringe sense to unique {\tt sin}-tree with fringe distribution $\fpm_{\Poi, 1}$. 

\begin{thm}
\label{thm:lwc-meso}
     Consider the sequence of random trees $\set{\cT_n:n\geq 2}$ in the mesoscopic regime with parameter $\beta \in (0,1)$. Then $\set{\cT_n:n\geq 2}$ converges in probability in the extended fringe sense (Def.~\ref{def:local-weak}~\eqref{it:fringe-b}) to the unique infinite {\tt sin}-tree with fringe distribution $\fpm_{\Poi, 1}$. This implies in particular, the degree distribution satisfies the asymptotics for $k\geq 1$,
     \[
\frac{N_k(n)}{n} \convp \frac{e^{-1}}{(k-1)!} \text{ as } n\to\infty,
\]
\end{thm}

\begin{rem}
    The above result thus implies that in the mesoscopic regime, the local weak limit of the sequence of trees, for any choice of $\beta \in (0,1)$ is the same as the local weak limit of a the uniform random tree (alternatively a uniformly sampled random spanning tree from the complete graph) as shown in \cite{grimmett1980random} and further insensitive to the choice of $\beta$. 
\end{rem}

\subsection{\textcolor{black}{Results on global functionals}}

\textcolor{black}{For a rooted tree $T$ with root $\rho$, we denote the height of $T$ by $h(T)$, where by definition
\begin{align*}
    \cH(T)=\sup\{d_T(v,\rho):v \in T\},
\end{align*}
where $d_T$ denotes the (graph) distance of $v$ from $\rho$. From here on, throughout the paper, we denote by $H_n$ the height $\cH(\cT_n)$ of the tree $\cT_n$. {As new incoming vertices prefer to connect to late arrived individuals (as quantified by $j(n)$), it is natural to expect that the height of the obtained trees should be stochastically larger than the case $j(n)=1$, in which case it turns out that the tree has height about $\log n$, see e.g., Pittel \cite{pittel1994note}. But how \emph{tall} does it get? In this section, we answer this question for the two regimes we consider. We also establish a phase transition on the global geometry of our trees in the mesoscopic regime.}
}

Before diving into the results, let us describe the non-obvious findings in each regime starting with the macroscopic regime. Recall the branching process $\BP_\theta$ that arose to describe the local weak limit in this regime.  Note that this branching process can go extinct with positive probability, namely $\pr(|\BP_{\theta}(\infty)| < \infty) >0$. Here we write $|\BP_\theta(\infty)|$ for the final size of the process.   

For each $n\geq 1$, let $T_n = \inf\set{t: |\BP_{\theta}| = n}$ for the time for this process to get to size $n$ with the convention $T_n = \infty$ if $|\BP_\theta(\infty)| < n$, i.e., if the branching process dies out before it grows to size $n$. Write $\cT_{n,\theta}^{\BP} = \BP_{\theta}(T_n)$ to denote the random tree representing the genealogical information of the individuals present in the population at this time and consider the sequence of trees $\set{\cT_{n,\theta}^{\BP}: n\geq 1}$. This sequence of trees is {\bf growing} only on the set $\BP_\theta(\infty) = \infty$ and thus, this sequence of trees is different in distribution from $\set{\cT(n):n\geq 1}$. The next proposition follows easily from classical results and proof techniques in Kingman and Pittel \cites{kingman1975first,pittel1994note} and can be found in Section \ref{sec:proofs}. We need some additional notation. Define \textcolor{black}{for any $\lambda>1$,}
\begin{equation}
    \label{eqn:def-ht-con}
    \phi(\lambda) = \frac{1-\theta^\lambda}{\lambda(1-\theta)}, \qquad \mu(a) = \inf_{\lambda >1}{\phi(\lambda)e^{\lambda a}}, \qquad \kappa(\theta) = \sup\set{a: \mu(a) <1}. 
\end{equation}

\begin{prop}
\label{prop:bp-height}
  On the set $\set{|\BP_\theta(\infty)| = \infty}$,  the height of the branching process genealogies $\cH(\cT_{n,\theta}^{\BP})$ satisfy $\cH(\cT_{n,\theta}^{\BP})/\log{n} \convas [\kappa(\theta)]^{-1}$ where $\kappa(\theta) $ is as in \eqref{eqn:def-ht-con}. 
\end{prop}

The next result shows that, not just local weak limits, but even the height of the original tree process has the same asymptotics as the branching process driven system.  

\begin{thm}
\label{thm:ht-macro}
The height \textcolor{black}{sequence $(H_n)_{n \geq 1}$} of the original growing trees $\set{\cT(n):n\geq 1}$ in the macroscopic regime satisfy $H_n/\log{n} \convas [\kappa(\theta)]^{-1}$ where $\kappa(\theta)$ is as in Prop. \ref{prop:bp-height}.
\end{thm}
{Thus, in the macroscopic regime, the effect of $j(n)$ is not strong enough to increase the height (in order) from the $j(n)=1$ case.}

\begin{rem}
    The convergence of the properly scaled height was already proven in \cite{devroye2012depth}. The main contribution of the above result is showing the relationship between the limit constants derived in \cite{devroye2012depth} and the so-called ``first birth'' problem in continuous time branching processes. The paper \cite{devroye2012depth} focuses on the height and depth of a general class of random tree growth models;  Section \ref{sec:disc} shows how the local weak limits of these models can be derived  and described using associated continuous time branching processes, using the proof techniques in this paper. 
\end{rem}

Now recall that Theorem \ref{thm:lwc-meso} implies that in the mesoscopic regime,  the local weak limit is insensitive to the choice of the parameter $\beta$ and is in fact the same as the limit of a uniform random tree. The next results shows that the story is quite different for the height and in fact for the global scaling of distances. 
\begin{thm}\label{thm:ht_meso}
    \textcolor{black}{Consider the mesoscopic regime with $j(n)=n-n^{\beta}$ for a fixed paramter $\beta\in (0,1)$. The height $H_n$ of the tree $\cT_n$ satisfies $H_n/n^{1-\beta}\convp \frac{2}{1-\beta}$ as $n \to \infty$.}
\end{thm}
{Thus, we note that in the mesoscopic regime, the effect of $j(n)$ is strong enough to completely change the height behavior from the $j(n)=1$ case.}

{
We delve deeper into the global structure in the mesoscopic regime next. The following three theorems establish a phase transition in the global structure of the graph in the mesoscopic regime at $\beta=1/2$. Informally, for $\beta<1/2$, the global structure approximates a line. For $\beta>1/2$ the tree contains many disjoint long paths, close in length to the height of the tree, consisting of the ancestral lines of the youngest vertices. For $\beta=1/2$, the ancestral lines of the youngest vertices all have length that approximate the height of the tree, but now they meet or \emph{coalesce} at a  height that is random on the leading order. Observe that this phase transition is not observed in the local limit, nor in the height asymptotics in Theorems~\ref{thm:lwc-meso} and~\ref{thm:ht_meso}. For the relevant relatively standard definitions of Gromov-Hausdorff topology that are used in the next theorems, we refer the reader to \cite{burago2001course}, \cite[Section 3]{gall2011scaling}. }

{
\begin{thm}\label{thm:ghp_meso}
In the mesoscopic regime when $\beta\in (0,1/2)$ we have 

\[\frac{1-\beta}{2} n^{-(1-\beta)}\mathcal{T}_n \convp L  \]
in the Gromov--Hausdorff topology, with $L$ a line segment of length $1$. 
\end{thm}

\begin{thm}\label{thm:meso_star}
In the mesoscopic regime when $\beta \in (1/2,1)$ we have, for any $k\ge 1$, for $\mathcal{T}^{(k)}_n$ the subtree of $\mathcal{T}_n$ spanned by the vertices $n, n-1, \dots, n-k+1$, 
\[\frac{1-\beta}{2} n^{-(1-\beta)}\mathcal{T}^{(k)}_n \convp S_k  \]
in the Gromov--Hausdorff topology, for $S_k$ a star graph with $k$ legs of length $1$. 
\end{thm}

\begin{thm}\label{thm:meso_not_crt}
    In the mesoscopic regime when $\beta=1/2$, for any $k\ge 2$ and for $\mathcal{T}^{(k)}_n$ the subtree of $\mathcal{T}_n$ spanned by vertices $n, n-1, \dots, n-k+1$, 
 \[\frac{1}{4} n^{-1/2}\mathcal{T}^{(k)}_n \to T_k  \]
    in the Gromov--Hausdorff topology, for $T_k$ a rooted $\mathbb{R}$-tree with $k$ leaves, all at height $1$, and its branchpoints at heights $X_1<\dots<X_{k-1}$, where, for $U_1,\dots,U_{k-1}$ independent $\operatorname{uniform}[0,1]$ random variables, we may define $X_{k-1}= U_{k-1}^{1/(4k(k-1))}$, and, recursively, for $2\le \ell \le k-1$, $X_{\ell-1}= X_\ell U_{\ell-1}^{1/(4\ell(\ell-1))}$. 
\end{thm}
In particular, observe that, although distances in $\mathcal{T}_n$ are of the order $n^{1/2}$ and the global structure is random on the leading order, a potential limit of $n^{-1/2}\mathcal{T}_n$ can almost surely be distinguished from the Brownian continuum random tree, as the latter does not contain a pair of leaves that realize the height of the tree and have their most recent common ancestor at a non-trivial height. A full scaling limit result is better postponed for a future work, see also the open problems in Section \ref{sec:open_probs}. But let us remark in passing that the limiting random fractals that should arise as such scaling limits, has the potential to give rise to novel universality classes in the world of random trees.}

\section{Discussion and extensions}
\label{sec:disc}

\subsection{Related work}
{The study of uniform random recursive trees (corresponding to $j(n)=1$ of our model) goes back to papers by Na and Rapoport \cite{na1970distribution}, Moon \cite{moon1974distance} and Gastwirth \cite{gastwirth1977probability}. We also refer the reader to the book of Drmota \cite{drmota2009random} and the survey by Smythe and Mahmoud \cite{smythe1995survey}. The study of depth properties of random recursive trees began with the papers of Devroye \cite{devroye1988applications} and Mahmoud \cite{mahmoud1991limiting}, who showed that the height of the vertex $n$ is about $\log n$, and Pittel \cite{pittel1994note} later showed that the height of the tree is about $e \log n$ by using a continuous time embedding.}

{As in our work, various generalizations of this model has been considered in the literature. The earliest and most notable being by Szymański \cite{szymanski1987nonuniform} who considered the case when a new incoming vertex sample a parent \emph{preferentially} with a probability proportional to its degree. This gives rise to \emph{plane-oriented recursive trees}, also known as \emph{preferential attachment trees}. Works of Mahmoud \cite{mahmoud1992distances} and Pittel \cite{pittel1994note} analyze the height of these trees, which again turns out to be $\log n$ in order. When an incoming new vertex is allowed to connect to more than one parent, the preferential choice gives rise to the well known \emph{preferential attachment model} of Barabasi and Albert \cite{barabasi1999emergence}. While when a new vertex samples more than one parent, but now uniformly at random, one obtains the so called \emph{random recursive directed acyclic graphs} or \emph{DAGs}, introduced by Devroye and Lu \cite{devroye1995strong}, and found application later in modeling circuits, see Tsukiji and Xhafa \cite{tsukiji1996depth}; Arya, Golin and Mehlhorn \cite{arya1999expected}.}

{In a different direction, Devroye, Fawzi and Fraiman \cite{devroye2012depth} introduce the model of \emph{scaled attachment random recursive trees}, or \emph{SARRTs} (see Section \ref{sec:SARRTs} for the precise definition), where the sampling of a random parent by the incoming new vertex is guided by a given probability distribution on $[0,1]$, and prove that the height of these trees are of order $\log n$. Various other generalizations have been considered in the literature, see e.g. \cite{banerjee2024co} for a model where the new incoming vertex uniformly samples a vertex in the current tree, and then goes up its ancestral line for a random amount and finally connects to where it ends up, \cite{addario2024random} for a model where the new incoming vertex connects to a uniformly sampled \emph{neighbor}, of a randomly sampled node in the existing network, and \cite{bhamidi2022community} for a model where the set of vertices of the tree is divided into two communities, and the new incoming vertex has to make probabilistic decisions on what type it assumes and what type it (uniformly) connects to. Also relevant are the two works \cite{BBDS04_macro} and \cite{BBDS04_meso} where the authors consider a model where incoming vertices actually prefer earlier arrived vertices to connect to, as a consequence of not having enough information on the current state of the network. We also refer the reader to the nice survey \cite{holmgren2017fringe} on fringe trees by Holmgren and Janson, and to Devroye's notes \cite{devroye1998branching} on branching processes.}

\subsection{Local weak limits of general SAARTs}\label{sec:SARRTs}
One goal of this paper is to develop new proof techniques for proving local weak convergence of random tree models to limiting infinite objects. While Theorem \ref{thm:lwc-macro} describes local weak convergence for the specific model occupying the majority of the paper, as mentioned in \textcolor{black}{Section \ref{sec:model},} in the macroscopic regime \textcolor{black}{our model} is a specific example of scaled attachment random recursive trees (SARRT) as formulated in \cite{devroye2012depth}. \textcolor{black}{The model is described as follows:} 
\begin{enumeratea}
    \item Fix a distribution $\mu$ on $[0,1]$. We will restrict to the setting where $\mu$ has a density $f_V$ with respect to the Lebesgue measure. 
    \item Let $V_1, V_2, \ldots, $ be \emph{i.i.d.} with distribution $\mu$. 
    \item Generate a sequence of random trees where recursively after constructing $\cT_n$, to construct $\cT_{n+1}$ generate a new vertex labelled $n+1$ and connect it to vertex labelled $\lfloor n V_n \rfloor$. 
\end{enumeratea}
Consider the function 
\[\lambda_V(x) = f_V(e^{-x}), \qquad x\geq 0. \]
Let $\cP_V$ denote a Poisson point process with intensity function $\lambda_V$. Let $\BP$ be a continuous time branching process started with one individual at time $t=0$ and offspring distribution $\cP_V$. Let $T_1\sim \exp(1)$ independent of $\BP_V$. Let $\fpm_{V}$ denote the distribution of $\BP_{V}(T_1)$, viewed as a random finite rooted tree on $\bbT$, where we retain only genealogical information between individuals in $\BP_{V}(T_1)$. 
\begin{thm}
\label{thm:lwc-saart}
     Consider the sequence of random trees $\set{\cT_n:n\geq 2}$ distributed as an SAART with attachment distribution $V$.  Then $\set{\cT_n:n\geq 2}$ converges in probability in the extended fringe sense (Def.~\ref{def:local-weak}~\eqref{it:fringe-b}) to the unique infinite {\tt sin}-tree with fringe distribution $\fpm_{V}$. This implies in particular, the degree distribution satisfies the asymptotics for $k\geq 1$,
     \[
\frac{N_k(n)}{n} \convp \pr(\cP_V(T_1) = k-1) \text{ as } n\to\infty,
\]
where $T_1\sim \exp(1)$ independent of $\cP_V(\cdot)$. 
\end{thm}


\subsection{Open questions}\label{sec:open_probs}
The goal of this paper is to understand the impact of lack of information on the state of the network in the evolution of the network. There are a host of unexplored directions. \textcolor{black}{Some of these include:}
\begin{enumeratea}
    \item Considering settings where instead of uniformly attaching to a vertex amongst those observed, one uses the degree distribution of the vertices available to make probabilistic decisions \textcolor{black}{in a preferential fashion.}
    \textcolor{black}{\item As outlined in Remark \ref{rem:future_work_general_j_nontree}, it is interesting to consider more general choices for the sequence $j(n)$, and the case where the network growth is governed by vertices coming in with more than one edge, i.e., the non-tree setting.}
    \item \textcolor{black}{Deriving} more detailed second order asymptotics for the fluctuation of the height about its mean. {The phase transition implied by Theorems \ref{thm:ghp_meso} and \ref{thm:meso_star} suggests that, in the mesoscopic regime, the second-order behaviour is different for $\beta<1/2$ and $\beta>1/2$: for $\beta<1/2$ the maximum is realized by the length of a single long branch, whereas for $\beta>1/2$ it results from a competition of many long branches, restricting the fluctuations in the latter case.}
    \item {Studying the limit in the Gromov--Hausdorff topology of $n^{-1/2}\cT_n$ for $\beta=1/2$ in the mesoscopic regime.  We conjecture that this model converges to a compact random fractal that is approximated by the limit $T_k$ of the subtree of $\cT_n$ spanned by $n, n-1,\dots,n-k+1$ by making $k$ large enough. In fact, when considering $j(n)=n-\lfloor \alpha n^{1/2} \rfloor$ for $\alpha\in \mathbb{R}$, we believe to find a family of novel limiting fractals that somehow interpolate between the regimes $\beta<1/2$ and $\beta>1/2$. }
\end{enumeratea}

\section{Local weak convergence and {\tt sin}-trees}
\label{sec:local-wll}

The goal of this section is to describe the standard notions of \emph{local weak convergence}~\cites{aldous-steele-obj,benjamini-schramm,van2023random}. The setting of trees is easier to grasp, and we largely follow Aldous~\cite{aldous-fringe}.


\subsection{Fringe decomposition for trees}
 \textcolor{black}{One of the core objects of this paper is the study of a sequence of growing random trees $\set{\cT_n:n\geq 1}$. Throughout we will write $\deg(v,\cT_n)$ for the degree of the vertex $v$ in the tree $\cT_n$ and write for the empirical degree counts,  
\begin{equation}
    \label{eqn:deg-count}
    N_k(n) = \sum_{v\in \cT_n} \ind\set{\deg(v,\cT_n) = k }, \qquad k\geq 1.
\end{equation}
When there is no scope for confusion, we will simplify the above notation to $\deg(v,n)$.} Recall that the space of finite rooted trees was denoted by $\bbT$.  For any $r\geq 0$ and $\bt\in \bbT$, let $B(\bt, r) \in \bbT$ denote the subgraph of $\bt$ of vertices within graph distance $r$ from $\rho_{\bt}$, viewed as an element of $\bbT$ and rooted again at $\rho_{\bt}$. 

 Given two rooted finite trees $\bs, \bt \in \bbT$,  say that $\bs \simeq \bt $ if  there exists a {\bf root
preserving} isomorphism between the two trees viewed as unlabelled graphs. Given two rooted trees $\bt,\bs \in \bbT$ 
, define the distance (\cite{benjamini-schramm},~\cite{van2023random}*{Equation 2.3.15})
\begin{align}
\label{eqn:distance-trees}
	d_{\bbT}(\bt,\bs):= \frac{1}{1+R^*}, \qquad \text{with } \qquad R^* =\sup\{r: B(\bt, r) \simeq B(\bs, r)
	\}.
\end{align}

Next, fix a tree $\bt\in \bbT$ with root $\rho = \rho_\bt$ and a vertex $v\in \bt$ at (graph) distance $h$ from the root.  Let $(v_0 =v, v_1, v_2, \ldots, v_h = \rho)$ be the unique path from $v$ to $\rho$. The tree $\bt$  can be decomposed as $h+1$ rooted trees $f_0(v,\bt), \ldots, f_h(v,\bt)$, where $f_0(v,\bt)$ is the tree rooted at $v$, consisting of all vertices for which there exists a path from the root passing through $v$. For $i \ge 1$, $f_i(v,\bt)$ is the subtree rooted at $v_i$, consisting of all vertices for which the path from the root passes through $v_i$ but not through $v_{i-1}$. 

\textcolor{black}{Let us denote the infinite-fold product of $\bbT$ as $\bbT^{\infty}$, i.e., it is the set of all sequences of the form $(\bt_1,\bt_2,\dots)$.} Call the map $(v,\bt) \leadsto \bbT^\infty$ where $v\in \bt$, defined via, 
\[F(v, \bt) = \left(f_0(v,\bt), f_1(v,\bt) , \ldots, f_h(v,\bt), \emptyset, \emptyset, \ldots \right),\]
as the fringe decomposition of $\bt$ about the vertex $v$. Call $f_0(v,\bt)$ the {\bf fringe} of the tree $\bt$ at $v$.
For $k\geq 0$, call $F_k(v,\bt) = (f_0(v,\bt) , \ldots, f_k(v,\bt))$ the {\bf extended fringe} of the tree $\bt$ at $v$ truncated at distance $k$ from $v$ on the path to the root.

Now consider the space $\bbT^\infty$. The metric in~\eqref{eqn:distance-trees} extends  to $\bbT^\infty$, \eg\ using the distance,
\begin{align}
\label{eqn:dist-inf}
	d_{\bbT^\infty}((\bt_0, \bt_1, \ldots),(\bs_0, \bs_1, \ldots)):= \sum_{i=0}^\infty \frac{1}{2^i} d_{\bbT}(\bt_i, \bs_i). 
\end{align}
We can also define analogous extensions to $\bT^k$ for finite $k$.  

Next, an element $\bfomega = (\bt_0, \bt_1, \ldots) \in \bbT^\infty$, with $|\bt_i|\geq 1$ for all $ i\geq 0$,  can be thought of as a locally finite infinite rooted tree with a {\bf s}ingle path to {\bf in}finity (thus called a {\tt sin}-tree~\cite{aldous-fringe}), as follows: Identify the sequence of roots of $\set{\bt_i:i\geq 0}$ with the integer lattice $\Zbold_+ = \set{0,1,2,\ldots}$, equipped with the natural nearest neighbor edge set, rooted at $\rho=0$. Analogous to the definition of extended fringes for finite trees, for any $k\geq 0$, write 
$F_k(0,\bfomega)= (\bt_0, \bt_1, \ldots, \bt_k)$. 



Call this the extended fringe of the tree $\bfomega$ at vertex $0$, till distance $k$, on the infinite path from $0$. Call $\bt_0 = F_0(0,\bfomega)$ the {\bf fringe} of the {\tt sin}-tree $\bfomega$. Now suppose $\prob$ is a probability measure on $\bbT^\infty$ such that, for $\TT:= (\bt_0(\TT), \bt_1(\TT),\ldots)\sim \prob$,  $|\bt_i(\TT)|\geq 1$ a.s.  $\forall~i\geq 0$. Then $\TT$ can be thought of as an infinite {\bf random} {\tt sin}-tree. 

Define a matrix $\vQ = (\vQ(\vs,\vt): \vs, \vt \in \bbT)$ as follows: suppose the root $\rho_{\vs}$ in $\vs$ has degree $\deg(\rho_{\vs}) \ge 1$, and let $(v_1,\ldots, v_{\deg(\rho_{\vs})})$ denote its children. For $1\leq i\leq {\deg(\rho_{\vs})}$, let $f(\vs, v_i)$ be the subtree below $v_i$ and rooted at $v_i$, viewed as as an element of $\bbT$. Write,
\begin{align}
\label{eqn:Q-def}
	\vQ(\vs,\vt):= \sum_{i=1}^{\deg(\rho_{\vs})} \ind\set{d_{\bbT}(f(\vs, v_i), \vt) = 0}. 
\end{align} 
Thus, $\vQ(\vs, \vt)$ counts the number of descendant subtrees of the root of $\vs$ that are  \textcolor{black}{(root preserving)} isomorphic 
to $\vt$. If $\deg(\rho_{\vs})=0$, define $\vQ(\vs, \vt)=0$. Now consider a sequence $(\bar \vt_0, \bar \vt_1, \dots)$ of trees in $\bbT$ such that $\vQ(\bar \vt_i, \bar \vt_{i-1}) \ge 1$ for all $i \ge 1$. Then there exists a unique infinite {\tt sin}-tree $\TT$ with infinite path indexed by $\Zbold_+$ such that $\bar \vt_i$ is the subtree rooted at $i$ for all $i \in \Zbold_+$. Conversely, it is easy to see, by taking $\bar \vt_i$ to be the union of (vertices and induced edges) of $\vt_0,\dots, \vt_i$ for each $i \in \Zbold_+$, that every infinite {\tt sin}-tree has such a representation. Following~\cite{aldous-fringe}, we call this the \emph{monotone representation} of the {\tt sin}-tree $\TT$.

\subsubsection{Convergence on the space of trees}
\label{sec:fringe-convg-def}
For $1\leq k\leq \infty$, let $\cM_{\pr}(\bbT^k)$ denote the space of probability measures on the associated space, metrized using the topology of weak convergence inherited from the corresponding metric on the space $\bbT^k$, see, e.g.,~\cite{billingsley2013convergence}. 
Suppose $\set{\cT_n}_{n\geq 1} \subseteq \bbT$ be a sequence of {\bf finite} rooted random trees on some common probability space (for notational convenience, assume $|\cT_n|= n$, or more generally $|\cT_n|\convas \infty$). For $n\geq 1$ and for each fixed $k\geq 0$,  the empirical distribution of fringes up to distance $k$ 
\begin{align}
\label{eqn:empirical-fringe-def}
	\fP_{n}^k:= \frac{1}{n} \sum_{v\in \cT_n} \delta\set{F_k(v,\cT_n)}. 
\end{align}
Thus, $\set{\fP_{n}^k:n\geq 1}$ can be viewed as a random sequence in  $\cM_{\pr}(\bbT^k)$. In the following, $\E[\fP_n^{0}]$ denotes the measure given by $\E[\fP_n^{0}](\vt):= \E[\fP_n^{0}(\vt)], \, \vt \in \bbT$.

\begin{defn}[Local weak convergence]
	\label{def:local-weak}
 Fix a probability measure $\varpi$ on $\bT$.
	\begin{enumeratea}
 \item \label{it:fringe-exp} Say that a sequence of trees $\set{\cT_n}_{n\geq 1}$ converges in {\bf expectation}, in the fringe sense, to $\varpi$, if \[\E[\fP_n^{0}] \to \varpi, \quad \text{ as } n\to\infty. \]
 Denote this convergence by $\TT_n\Efr \varpi$ as $n\to\infty$.
	    \item \label{it:fringe-a}  Say that a sequence of trees $\set{\cT_n}_{n\geq 1}$ converges in the probability sense, in the fringe sense, to $\varpi$, if \[\fP_n^{0} \probc \varpi, \quad \text{ as } n\to\infty. \]
	Denote this convergence by $\TT_n\probfr \varpi$ as $n\to\infty$.
	    \item \label{it:fringe-b} Say that a sequence of trees $\set{\cT_n}_{n\geq 1}$ converges in probability, in the {\bf extended fringe sense}, to a limiting infinite random {\tt sin}-tree $\TT_{\infty}$ if for all $k\geq 0$ one has
	  \[\fP_n^k \probc \prob\left(F_k(0,\TT_{\infty}) \in \cdot \right), \qquad \text{ as } n\to\infty. \]
	Denote this convergence by $\TT_n\probcrf \TT_{\infty}$ as $n\to\infty$.
	\end{enumeratea}
\end{defn} 
In an identical fashion, one can define notions of convergence in distribution or almost surely in the fringe, respectively extended fringe sense. 
Letting $\varpi_{\infty}(\cdot) = \pr(F_0(0, \cT_{\infty}) = \cdot)$ denote the distribution of the fringe of $\cT_\infty$ on $\bT$, convergence in (c) above clearly implies convergence in notion (b) with $\varpi =\varpi_{\infty}(\cdot) $.
If the limiting distribution $\varpi$ in (b) has a certain `stationarity' property (defined below), \emph{convergence in the fringe sense implies convergence in the extended fringe sense} as we now describe. Under an ``extremality'' condition of the limit object in (a), convergence in expectation implies convergence in probability as in (b). We need the following definitions.

\begin{defn}[Fringe distribution~\cite{aldous-fringe}] \label{fringedef}
	Say that a probability measure $\varpi$ on $\bbT$ is a fringe distribution if 
	\[\sum_{\vs} \varpi(\vs) \vQ(\vs, \vt) = \varpi(\vt), \qquad \forall~\vt \in \bbT. \]
\end{defn}
It is easy to check that the space of fringe distributions $\cM_{\pr, \fringe}(\bbT) \subseteq \cM_{\pr}(\bT)$ is a convex subspace of the space of probability measure on $\bT$ and thus one can talk about extreme points of this convex subspace. The following fundamental theorem is one of the highlights of~\cite{aldous-fringe}. 
\begin{thm}[\cite{aldous-fringe}]
\label{thm:aldous-efr-pfr}
    Fix a fringe distribution $\varpi \in \cM_{\pr, \fringe}(\bbT)$. Suppose a sequence of trees $\set{\cT_n:n\geq 1}$ converges in the expected fringe sense $\cT_n \Efr \varpi$ as $n\to\infty$. If $\varpi$ is extremal in the space of fringe measures, then the above convergence in expectation automatically implies $\cT_n \probfr \varpi$. 
\end{thm}
The advantage of this result is that for proving convergence in the probability fringe sense, at least under the extremality of the limit object, dealing with expectations is enough. 
The next result shows that convergence in the probability fringe sense often automatically implies convergence to a limit infinite {\tt sin}-tree. We need one additional definition. For any fringe distribution $\varpi$ on $\bbT$, one can uniquely obtain the law $\varpi^{EF}$ of a random {\tt sin}-tree $\TT$ with monotone decomposition $(\bar\bt_0(\TT), \bar\bt_1(\TT),\ldots)$ such that for any $i \in \Zbold_+$, any $\bar \vt_0, \bar \vt_1, \dots$ in $\bbT$,
\begin{align}\label{ftoef}
\varpi^{EF}((\bar\bt_0(\TT), \bar\bt_1(\TT),\ldots, \bar\bt_i(\TT)) = (\bar \vt_0,\vt_1,\dots,\vt_i)) := \varpi(\vt_i) \prod_{j=1}^{i}Q(\vt_i,\vt_{i-1}),
\end{align}
where the product is taken to be one if $i=0$. The following Lemma follows by adapting the proof of~\cite{aldous-fringe}*{Propositions 10 and 11}, and the proof is omitted.

\begin{lem}\label{ftoeflemma}
Suppose a sequence of trees $\set{\cT_n}_{n\geq 1}$ converges in probability, in the fringe sense, to $\varpi$. Moreover, suppose that $\varpi$ is a fringe distribution in the sense of Definition~\ref{fringedef}. Then $\set{\cT_n}_{n\geq 1}$ converges in probability, in the extended fringe sense, to a limiting infinite random sin-tree $\TT_{\infty}$ whose law $\varpi^{EF}$ is uniquely obtained from $\varpi$ via~\eqref{ftoef}.
\end{lem}

Fringe convergence and extended fringe convergence imply convergence of \textcolor{black}{local} functionals, such as the degree distribution. For example, letting $\cT_{\varpi} \sim \varpi$ with root denoted by $0$ say, convergence in notion~\eqref{it:fringe-a} in particular implies that for any $k\geq 0$ \textcolor{black}{(assuming $\cT_n$ is a tree on $n$ vertices)}, 
\begin{align}
\label{eqn:deg-convg-fr}
\frac1n\cdot\#\set{v\in \TT_n: \deg(v) = k+1} \convp \prob(\deg(0,\cT_{\varpi})=k).
\end{align}
However, both convergences give much more information about the asymptotic properties of $\set{\cT_n:n\geq 1}$ beyond its degree distribution. \textcolor{black}{The interested reader is encouraged to have a look at \cite{aldous-fringe}.}

\section{Proofs of local results}
\label{sec:proofs}
\subsection{Local weak convergence}

We approximate the discrete tree using a continuous time branching process (CTBP). In this process, each vertex gives birth to children, and the time instants marking the child births follow a point process on the real line that is a Poisson process $\mathcal{N}_{\theta}$, and the processes for different vertices are independent. The process $\mathcal{N}_{\theta}$ has the distribution of a constant rate $1/(1-\theta)$ Poisson process on the interval $(0,\log(1/\theta))$. We start by collecting some simple properties of this branching process $\BP_{\theta}$. 

\begin{lem}[Malthusian rate of growth]
\label{lem:malthus}
    Consider the branching process $\BP_\theta$. Then 
    \begin{enumeratea}
        \item The Malthusian rate of growth of the branching process $\BP_\theta$ namely $\lambda(\theta) =1$. 
        \item The offspring distribution satisfies the $X\log{X}$ condition. In particular $e^{-t}\BP_\theta(t)\stackrel{a.s., \bL^1}{\longrightarrow} W_\theta$ where $W_\theta >0$ on the event of survival of the branching process namely  $\set{|\BP_\theta(\infty)| = \infty}$. 
    \end{enumeratea}
\end{lem}
\begin{proof}
Recall the offspring point process $\cN_\theta$ that drives the branching process.  To prove (a) we need to find the unique $\lambda$ such that $\phi(\lambda):= \E(\cP_\theta(T_\lambda)) = 1$    where $T_\lambda\sim \exp(\lambda)$ independent of $\cN_\theta$. By construction, 
\begin{align*}
    \phi(\lambda) &= \int_0^{\log{\theta^{-1}}} \frac{x}{1-\theta} \lambda e^{-\lambda x}dx + \frac{\log{\theta^{-1}}}{1-\theta} \pr(T_\lambda > \log{\theta^{-1}}) =\frac{1-\theta^\lambda}{\lambda(1-\theta)}, \\
\end{align*}
\textcolor{black}{and equating the RHS above to $1$ and solving for $\lambda$, we obtain $\lambda=1$. Also, $\lambda = 1$ is the unique solution as the function $\phi(\lambda)$ is non-increasing, which can verified by computing the derivative. This completes the proof of the lemma. }

Part(b) means we need to show 
\begin{equation}
    \E(\cN_\theta(T_\lambda))\log^+(\cN_\theta(T_\lambda))) < \infty.
    \label{eqn:cntheta}
\end{equation} However \[\cN_\theta(T_\lambda)) \stod \cN_\theta(\log(1/\theta)),\] where $\stod$ represents stochastic domination. Since $\cN_\theta(\log(1/\theta))$ has a Poisson distribution, thus all moments exist. This completes the proof of \eqref{eqn:cntheta}. 
\end{proof}

\subsection{Proof of local weak convergence in the Macroscopic regime: Theorem \ref{thm:lwc-macro}:} \label{sec:proof-lwc-macro}

 The 
  proof follows by modifying the arguments of Theorem 4.1 in \cite{BBDS04_macro}. For completeness, we state the main steps in the proof. For a vertex $n\geq 1$ and $k\geq 0$, define $T_k^n = \inf\set{\ell: \deg(n,\ell) = k+1}\geq n$ to be the first time vertex $n$ has its $k$-th child (and $T_0^n = n$) in the discrete time \textcolor{black}{tree} process $\{\TT(j) : j \ge 1\}$. For $k\geq 1$, we label the $k^{th}$ vertex entering the \emph{fringe} of vertex $n$ (the progeny tree \textcolor{black}{or the subtree} associated with vertex $n$) by $\rho_k^n$ i.e \begin{equation*}
    \rho_k^n = \inf\set{m> \rho_{k-1}^n: \text{ vertex }m \text{ attaches to one of } \rho_0^n,\rho_1^n,\dots, \rho_{k-1}^n}\end{equation*}  and $\rho_0^n = n$. For $j\geq 0,k\geq 1$, let $T^n_{j,k} = T^{\rho_j^n}_k$ be the birth time of $k$th child of $j$th vertex in the fringe of vertex $n$ (note that $T^n_{0,k} = T^n_k$ and $T^n_{j,0} = \rho^n_j$), and  \begin{align*}
    \tau_{j,k}^n = \log\left(\frac{T^n_{j,k}}{T^n_{j,k-1}}\right)
\end{align*}be the inter-arrival time in $\log$-scale.
One can think of the above quantity as the `continuous time' (if we encode vertex arrivals by epochs of a Yule process) elapsed for the degree of the $j$th \textcolor{black}{vertex in the fringe of vertex $n$, or simply the $j$th} descendant of vertex $n$, to increase from $k-1$ to $k$. These form a point process of reproduction times of vertices in the fringe of vertex $n$. The following \textcolor{black}{Proposition} shows that the inter-arrival times for these point processes weakly converge to those of the Poisson process $\mathcal{N}_\theta$ defined in section \ref{subsec:results-macro} as $n \to \infty$. Further, the point processes associated with distinct vertices in the fringe become asymptotically independent.

Let $\sE_{0\leadsto 1},\sE_{1\leadsto 2},\dots $ denote the inter-arrival times of the Poisson process $\mathcal{N}_{\theta}$.

\begin{prop}\label{prop:degree-times-n} 
Let $\left(\left(\sE_{j,k}\right)_{k\geq 1}: j\geq 0\right)$ be \emph{iid} copies of $\left(\sE_{k-1\leadsto k}\right)_{k\geq 1}$. Then, as $n \to \infty$,\begin{align*}
 \left(\left(\tau_{j,k}^n\right)_{k\geq 1}: j\geq 0\right) \convd \left(\left(\sE_{j,k}\right)_{k\geq 1}: j\geq 0\right).
\end{align*}
\end{prop}

We now present a direct consequence of the Proposition \ref{prop:degree-times-n}. The following corollary is the main tool used to prove the local weak convergence. Let $U$ be a uniform $U[0,1]$ random variable independent of the network process $\set{\cT(n)}_{n\geq 1}$.

\begin{cor}
\label{cor:degree-times-uniform}
   Let $\set{\left(\sE_{j,k}\right)_{k\geq 1}:j\geq 0}$ be \emph{iid} copies of $\left(\sE_{k-1\leadsto k}\right)_{k\geq 1}$ and $T_1$ be an exponential random variable of rate $1$ independent of $\set{\left(\sE_{j,k}\right)_{k\geq 1}:j\geq 0}$. Then 
    \begin{equation*}
        \left(\log n - \log \lceil n U\rceil,\left(\tau_{j,k}^{\lceil nU \rceil}\right)_{k\geq 1}: j\geq 0 \right) \convd \left(T_1,\left(\sE_{j,k}\right)_{k\geq 1}:j\geq 0 \right).
    \end{equation*}
\end{cor}

We now complete the proof of Theorem \ref{thm:lwc-macro} using the above lemma. 

\begin{proof}[Proof of Theorem \ref{thm:lwc-macro}]
\textcolor{black}{For any $t \in \mathbb{R}_+$ and $\mathbf{w} \in \mathbb{R}_+^{\mathbb{N}_0 \times \mathbb{N}}$, one can construct a unique rooted tree $\mathscr{T}(t,\mathbf{w}) \in \mathbb{T}$ via the following temporal evolution: (a) the root $\emptyset$ is born at time $0$, (b) for any $j \ge 0, k \ge 1$, $w_{j,k}$ denotes the age of the $j$th born vertex in the tree when its $k$th child is born (with root being the $0$th born vertex), (c) retain only those vertices with birth times $\le t$. For $0 \le j < |\mathscr{T}(t,\mathbf{w})|$, denote by $\tau_j(t,\mathbf{w})$ the birth time of the $j$th born vertex in $\mathscr{T}(t,\mathbf{w})$, and let $n(t,\mathbf{w}) := \sup\{j \ge 0 : \tau_j(t,\mathbf{w}) \le t\}$. Define $\mathscr{W} := \{(t,\mathbf{w}) \in  \mathbb{R}_+ \times \mathbb{R}_+^{\mathbb{N}_0 \times \mathbb{N}} : n(t,\mathbf{w}) < \infty, \, \tau_j(t,\mathbf{w}) \neq t \text{ for any } 0 \le j \le n(t,\mathbf{w})\}$. Then, it is clear that $\mathscr{T}: \mathscr{W} \rightarrow \mathbb{T}$ is continuous \textcolor{black}{with respect to the associated product topology.}}

Now, consider the evolution of the network process $\set{\cT(n)}_{n\geq 1}$ in the logarithmic time scale, namely, assign the birth time of $0$ to $v_0,v_1$ and $\log j$ to $v_j$ for $j \ge 2$. For any $n \ge 2$ and $2 \le \ell \le n$, the \emph{fringe} of vertex $\ell$ in $\cT(n)$ can be expressed as $\mathscr{T}(\log n - \log \ell, \mathbf{w}^{(\ell)})$, where $w^{(\ell)}_{j,k} =\tau_{j,k}^\ell = \log\left(\frac{T^\ell_{j,k}}{T^\ell_{j,k-1}}\right)$, $j \ge 0, k \ge 1$. In particular, the fringe of a uniformly chosen vertex is given by
$\mathscr{T}(\log n - \log \lceil n U\rceil, \mathbf{w}^{(\lceil n U\rceil)})$. By Corollary ~\ref{cor:degree-times-uniform}, $(\log n - \log \lceil n U\rceil, \mathbf{w}^{(\lceil n U\rceil)}) \convd (T_1, \mathbf{w}^*)$ where $w^*_{j,k} = \sE_{j,k}$, $j \ge 0, k \ge 1$. Moreover, we have $\prob\left((T_1, \mathbf{w}^*) \in \mathscr{W}^c\right)=0$.
Hence, by the continuous mapping theorem, $\mathscr{T}(\log n - \log \lceil n U\rceil, \mathbf{w}^{(\lceil n U\rceil)}) \convd \mathscr{T}(T_1, \mathbf{w}^*)$, which has the law $\fpm_{\theta}$. This implies the local weak convergence in the expected fringe sense.

The convergence of the (expected) empirical degree distribution is an immediate consequence of this local weak convergence.
\end{proof}

The rest of the section is dedicated to the proof of the Proposition \ref{prop:degree-times-n}. 
\begin{proof}[Proof of Proposition \ref{prop:degree-times-n}]
     Throughout the proof, we fix $K \in \mathbb{N}$. We start by computing the joint density of $\sE_{K}= (\sE_{0\leadsto 1},\sE_{1 \leadsto 2},\dots,\sE_{K-1\leadsto K})$. \textcolor{black}{Recall $c_{\theta}=\log \theta^{-1}$.} Let $h(x) = \frac{\ind\set{x \leq c_\theta}}{1-\theta}$ and $H(x) = \int_0^x h(y)dy = \frac{\min(x,c_\theta)}{1-\theta}$. For any $\mvx = (x_1,x_2,\dots,x_K)\in \bR^K_+$ and $1 \le k \le K$, let $\sigma_i(\mvx) = \sum\limits_{l=1}^i x_l$ and $\sigma_0(\mvx) = 0$.
Using these functions and the definition of the process $\mathcal{N}_{\theta}$, it is easy to verify that the joint density of the inter-arrival times $\sE_{K}= (\sE_{0\leadsto 1},\sE_{1 \leadsto 2},\dots,\sE_{K-1\leadsto K})$ is given by 

\textcolor{black}{
\begin{align}\label{eqn:joint-density}
    \fF_K(\mvx) =  \left(\prod\limits_{i=1}^K h(\sigma_i(\mvx)) \right)\exp\left(-H(\sigma_K(\mvx))\right) = \frac{\ind\set{\sigma_{K}(\mvx) \leq c_{\theta}}}{(1-\theta)^K} \exp(-H(\sigma_K(\mvx)) , \ \mvx \in \bb{R}_+^K.
\end{align}}
We also note that the density $\fF_K$ is continuous almost surely on $\bR_+^K$ w.r.t.~Lebesgue measure. \textcolor{black}{Next,} we prove \begin{align*}
        \tau^n_{1:K} := \left(\tau^n_{j,k}:0\leq j\leq K, 1 \leq k\leq K\right) \convd \sE_{1:K} := \left(\sE_{j,k}:0\leq j\leq K, 1 \leq k\leq K\right). 
    \end{align*} To this extent, we argue that \begin{align}\label{eqn:exp-convergence}
        \lim\limits_{n\to \infty}\E\left(f(\tau^n_{1:K})\right) = \E\left(f(\sE_{1:K})\right)
    \end{align} for any continuous function $f:\bR_+^{K^2+K}\to \bR$ with compact support $\cS \subseteq [0,L]^{K^2+K}$ for some $L \geq 0$. The proof of the above follows from \textcolor{black}{ the following} two steps:
    \begin{itemize}
        \item[(i)] We first express $\E\left(f(\tau^n_{1:K})\right)$ as an integral over $\bR_+^{K^2+K}$.
        \item[(ii)] \textcolor{black}{We then use} a sequence of approximations to the law of $\tau^n_{1:K}$ to prove~\eqref{eqn:exp-convergence}.
    \end{itemize}

\noindent \textbf{Step} (i): We first develop some notation to express the distribution of $\tau^n_{1:K}$. Recall $(\rho_j^n:j\geq 0)$ and $(T_{j,k}^n: j,k\geq 0)$ defined at the beginning of Section~\ref{sec:proof-lwc-macro}.

     Note that the values of $(\rho_j^n:0\leq j\leq K)$ are determined by the values of $(T_{j,k}-T_{j,k-1}: 0\leq j\leq K, 1 \leq k \leq K)$. Thus, for $\mvx \in \bR^{K^2+K}_+$, one can define $(\rho_j^n(\mvx):0\leq j\leq K)$, the values of $(\rho^n_j:0\leq j\leq K)$ on the event $E_{\mvx} = \set{T^n_{j,k}-T^n_{j,k-1} = \lfloor x_{j,k}\rfloor :0\leq j\leq K, 1\leq k\leq K}$, as follows:
    \begin{itemize}
    \item \textcolor{black}{$\rho_0^n(\mvx) = T^n_{0,0}(\mvx) = n$} and $T^n_{0,k}(\mvx) = \rho_0^n(\mvx) + \sum\limits_{l=1}^k \lfloor x_{0,l}\rfloor $ for $1 \le k\leq K$.
    \item Given $\set{T_{j,k}^n(\mvx):j\leq i,1\leq k\leq K}$ for $0 \le i < K$, define \begin{align*}
        \rho^n_{i+1}(\mvx) &= \text{the }(i+1)^{th} \text{ smallest value in }\set{T_{j,k}^n(\mvx):0\leq j\leq i,1\leq k\leq K},\\
        T_{i+1,k}^n(\mvx) &=  \rho^n_{i+1}(\mvx) + \sum\limits_{l=1}^k \lfloor x_{i+1,l}\rfloor, \ \textcolor{black}{T^n_{i+1,0} = \rho^n_{i+1}(\mvx)}.
    \end{align*}
\end{itemize} 
Also, define $\tau^n_{j,k}(\mvx) = \log\left(\frac{T_{j,k}^n(\mvx)}{T_{j,k-1}^n(\mvx)}\right)$ and $\tau^n_{1:K}(\mvx) = \left(\tau^n_{j,k}(\mvx):0\leq j\leq K, k\leq K\right)$ - the values of inter-arrival times in $\log$-scale on the event $E_{\mvx}$. 

 On the event $E_{\mvx}$, an incoming vertex  attaches to a vertex in the fringe of vertex $n$ only at times $
    \cA_{\mvx} = \set{T^n_{j,k}(\mvx):0\leq j\leq K, 1\leq k\leq K }$. Also, let $
    N_{\mvx} =\sup \cA_{\mvx}$ and $
    \cI_{\mvx} = \set{n+1,n+2,\dots,N_{\mvx}}\setminus\cA_{\mvx}$ be the times when the fringe is inactive, meaning there are no attachments.
    
In the discrete-time process $\set{\cT(n):n\geq 1}$, an incoming vertex can only attach to a single existing vertex in the network. So, we say $\mvx \in \bR^{K^2+K}$ is \emph{valid} if all $T_{j,k}^n(\mvx)$ are distinct for $0\leq j\leq K$ and $1\leq k\leq K$. 

For $0\leq j\leq K$ and $i \in \mathbb{N}$, { let $p^n_j(i,\mvx)$ denote the probability that an incoming vertex at time $i$ attaches to the $j^{th}$ vertex in the fringe of vertex $n$ given the history up to time $i-1$ on the event $\set{T^n_{j,k}-T^n_{j,k-1} = \lfloor x_{j,k}\rfloor \  \forall \ k \text{ such that } T_{j,k}^n(\mvx) < i}$. Explicitly, we have  }
\begin{align}\label{eqn:prob-step}
    p^n_j(i,\mvx) &= \ind\set{i\leq T_{j,K}^n(\mvx)}
        \frac{\ind\set{i\theta \leq \rho_{j}^n(\mvx)}}{\lfloor i(1-\theta)\rfloor}.
\end{align} 
\textcolor{black}{Define,} 
\begin{align}
\label{eqn:prob-Ex}
    \vP_{0,n}(\mvx) := \pr(E_{\mvx}) &\textcolor{black}{= \ind\set{\mvx \text{ is \emph{valid}}} \prod_{i\in \cI_{\mvx}} \left(1-\sum\limits_{j :\rho^n_j(\mvx) <i} p_j^n(i,\mvx)\right) . \prod\limits_{j,k\leq K}p_{j}^n\left(T_{j,k}(\mvx),\mvx\right)}\notag\\
    &= \ind\set{\mvx \text{ is \emph{valid}}} \prod_{i\in \cI_{\mvx}} \left(1-\sum\limits_{j=0}^K p_j^n(i,\mvx)\right) . \prod\limits_{j,k\leq K}p_{j}^n\left(T_{j,k}(\mvx),\mvx\right).
\end{align} This concludes step (i).

\vspace{0.5cm}

\noindent \textbf{Step} (ii): We now approximate $\vP_{0,n}(\mvx)$ using the following sequence of approximations. \begin{enumerate}
    \item We replace the terms in the first product in~\eqref{eqn:prob-Ex} by exponentials to get $\vP_{1,n}(\mvx)$ where \begin{align}\label{eqn:prob-ex-1}
    \vP_{1,n}(\mvx) = \ind\set{\mvx \text{ is \emph{valid}}} \prod\limits_{i\in \cI_{\mvx}} \exp\left(-\sum\limits_{j=0}^K p_j^n(i,\mvx)\right)  . \prod\limits_{j,k\leq K}p_{j}^n\left(T_{j,k}^n(\mvx),\mvx\right).
\end{align}
\item We next replace the $\lfloor i(1-\theta)\rfloor$ with $i(1-\theta)$ in the denominator of ~\eqref{eqn:prob-step} to get $\vP_{2,n}(\mvx)$ with \begin{align}\label{eqn:prob-ex-2}
\vP_{2,n}(\mvx) = \ind\set{\mvx \text{ is \emph{valid}}} \prod\limits_{i\in \cI_{\mvx}} \exp\left(-\sum\limits_{j=0}^K q_j^n(i,\mvx)\right)  \cdot \prod\limits_{j,k\leq K}q_{j}^n\left(T_{j,k}^n(\mvx),\mvx\right),
\end{align} 
where 
\begin{align}\label{eqn:q_j-defn}
    q^n_j(i,\mvx) &= \ind\set{i\leq T_{j,K}^n(\mvx)}\frac{\ind\set{i\theta \leq \rho_{j}^n(\mvx)}}{i(1-\theta)}.
\end{align}

\item  Finally, we approximate the sums of type $\sum\limits_{i=m_1}^{m_2}\frac{1}{i}$ with $\log\frac{m_2}{m_1}$ appearing in the exponent of the first term in~\eqref{eqn:prob-ex-2} to get $\vP_{3,n}(\mvx)$ where \begin{align}\label{eqn:prob-ex-3}
        \vP_{3,n}(\mvx) = \ind\set{\mvx \text{ is \emph{valid}}} \exp\left(-\sum\limits_{j=0}^K r_j^n(\mvx)\right)  . \prod_{j=0}^{K}\prod_{k=1}^K q_{j}^n\left(T_{j,k}^n(\mvx),\mvx\right).
    \end{align} where \begin{align}\label{eqn:r_j-defn}
        r^n_j(\mvx) = H\left(\log\left(\frac{T_{j,K}^n(\mvx)}{T_{j,0}^n(\mvx)}\right)\right) 
    \end{align} 
\end{enumerate} The following lemma bounds the errors made during each approximation. The proof of the lemma closely follows the proof of Lemma 5.3 in \cite{BBDS04_macro}, so \textcolor{black}{we} skip the proof.  
\begin{lem}\label{lem:bigggg-lemma}\textcolor{black}{Let $\cC$ be a bounded subset in $R_{+}^{K^2+K}$. Then we have 
\begin{align*}
    \sup_{i=0,1,2} \, \sup_{\mvx \in \cC}\left|\vP_{i,n}(\mvx) - \vP_{i+1,n}(\mvx) \right|\leq  \frac{C}{n}
\end{align*}for some constant $C \geq 0$ depending of the set $\cC$ and $K$.}
\end{lem}

\noindent{\bf Completing the Proof of Proposition~\ref{prop:degree-times-n}:} The arguments in the proof are very similar to the proof of Proposition 5.1 in \cite{BBDS04_macro}. We mention the main details for \textcolor{black}{the sake of} completeness. Since $\tau^n_{1:K}$ is a discrete random variable taking values in $\bN_+^{K^2+K}$, and using~\eqref{eqn:prob-Ex}, we have $$
    \E(f(\tau^n_{1:K})) = \int_{\bR_+^{K^2 + K}} f(\tau^n_{1:K}(\mvx)) P_{0,n}(\mvx)d\mvx. $$ Therefore, by Lemma~\ref{lem:bigggg-lemma}, we have \begin{align*}
    \left|\E(f(\tau^n_{1:K})) - \int_{\bR_+^{K^2+K}}f(\tau^n_{1:K}(\mvx))\vP_{3,n}(\mvx)\right| \leq \frac{C}{n}
\end{align*} for some constant $C \geq 0$ \textcolor{black}{depending only on $K$ and $L$ where the support of $f$ is $\mathcal{S} \subseteq[0,L]^{K^2 + K}$}. Thus, to prove~\eqref{eqn:exp-convergence}, it is sufficient to show \begin{align}\label{eqn:sufficient-condition}
    \lim\limits_{n\to \infty}\int_{\bR_+^{K^2+K}} f(\tau^n_{1:K}(\mvx)) \vP_{3,n}(\mvx)d\mvx  = \E\left(f(\sE_{1:K})\right).
\end{align}


To compute the integral on LHS of~\eqref{eqn:sufficient-condition}, consider the transformation \textcolor{black}{that transforms} the inter-arrival times \textcolor{black}{to an} exponential scale, where we make the change of variables $\mvx \to \mvy$ with { $y_{j,k} = \tau^n_{j,k}(\mvx) = \log\left(\frac{T_{j,k}^n(\mvx)}{T_{j,k-1}^n(\mvx)}\right), \ 0 \le j \le K, 1 \le k \le K$}. More precisely, for $\mvy \in \bR_+^{K^2 + K}$, we define the following:\begin{itemize}
    \item Let $\hat{\rho}^n_0(\mvy) = n$.
    \item For $i < K$, given $\set{\hat{\rho}^n_j(\mvy):0\leq j\leq i}$, let $\hat{\rho}^n_{i+1}(\mvy)$ to be the value of $\rho^n_{i+1}(\mvy)$ when $\set{T_{j,k}^n-T_{j,k-1}^n= \lfloor \hat{\rho}^n_j(\mvy)\exp\left(\sum\limits_{l=1}^{k-1}y_{j,l}\right)\cdot\left(\exp(y_{j,k})-1\right) \rfloor \ : 0\leq j\leq i,1\leq k\leq K}$. 
\end{itemize}  
We define $G_n: \mathbb{R}^{K^2 + K} \rightarrow \mathbb{R}^{K^2 + K}$ by $$G_n(\mvy) = \left(\hat{\rho}^n_j(\mvy)\exp\left(\sum\limits_{l=1}^{k-1}y_{j,l}\right)\cdot \left(\exp(y_{j,k})-1\right): 0\leq j\leq K, 1\leq k\leq K \right).$$ We compute the integral on LHS of~\eqref{eqn:sufficient-condition} after applying the transformation $\mvx \mapsto G_n(\mvy)$ \textcolor{black}{(which, as we will see, is the `approximate' inverse of the transformation $\mvy \mapsto \tau^n_{1:K}(\mvx)$)}. 
 As noted in \cite{BBDS04_macro}, the Jacobian of the transformation is given by \begin{align}\label{Jac}
    \fJ_n(\mvy) = \prod_{j=0}^{K}\prod_{k=1}^K \left[\hat{\rho}^n_j(\mvy)\exp\left({\sum_{i=1}^{k}y_{j,i}}\right)\right].
    \end{align} 
    Define $\hat{T}^n_{j,k}(\mvy) = T^n_{j,k}(G_n(\mvy))$ and $\hat{\tau}^n_{j,k}(\mvy)=\tau^n_{j,k}(G_n(\mvy))$. Then after the transformation, the integral on the LHS of~\eqref{eqn:sufficient-condition} is\begin{align*}
        \int_{\bR_+^{K^2+K}} f(\tau^n_{1:K}(\mvx)) \vP_{3,n}(\mvx)d\mvx = \int_{\bR_+^{K^2+K}} f(\hat{\tau}^n_{1:K}(\mvy)) \vP_{3,n}(G_n(\mvy)) \fJ_n(\mvy)d\mvy. 
    \end{align*} 
    Again, following the arguments in the proof of Proposition 5.1 of \cite{BBDS04_macro}, we see that the support of the integrand on the right-hand side of the above display is contained in $[0, \log\left(e^L + K\right)]^{K^2+K}$.
Recalling $q^n_j(i,\mvx)$ defined in~\eqref{eqn:q_j-defn} and $c_\theta = \log \theta^{-1}= -\log \theta$, we have for any $0 \le j \le K, 1 \le k \le K$,
\begin{align*}
    q^n_j(\hat{T}^n_{j,k}(\mvy),G_n(\mvy)) 
    &= \frac{h\left(\sigma_k(\hat{\tau}_j^n(\mvy))\right)}{\hat{T}^n_{j,k}(\mvy)}
\end{align*}
where $\hat{\tau}^n_{j}(\mvy) := (\hat{\tau}^n_{j,k}(\mvy): 1\leq k\leq K)$. Combining this with~\eqref{Jac}, we obtain
\begin{align}\label{eqn:bound-for-dct}
\begin{split}
&\left[\prod\limits_{j=0}^K\prod\limits_{k=1}^K q^n_j(\hat{T}^n_{j,k}(\mvy),G_n(\mvy))\right] \fJ_n(\mvy) = \left(\prod_{j=0}^{K}\prod_{k=1}^K h(\sigma(\hat{\tau}^n_{j}(\mvy)))\right) \left(\prod_{j=0}^{K}\prod_{k=1}^K \frac{\hat{\rho}^n_j(\mvy)\exp\left({\sum_{i=1}^{k}y_{j,i}}\right)}{\hat{T}_{j,k}^n(\mvy)}\right).
\end{split}
\end{align} 
Therefore, if  $\hat{\tau}_{1:K}^n(\mvy) \in \cS$, the LHS of~\eqref{eqn:bound-for-dct}, and hence $\vP_{3,n}(G_n(\mvy)) \fJ_n(\mvy)$, is uniformly bounded over $n$. Also, from the definition of $r^n_j$ in~\eqref{eqn:r_j-defn} and recalling the function $H$, we have $
    r^n_j(G_n(\mvy)) = H(\sigma_K(\hat{\tau}_j^n(\mvy)))$. Since $\hat{\rho}^n_j(\mvy) \geq n$, we have using the bounds on $\hat{T}^n_{j,k}(\mvy)$ obtained above that $\lim\limits_{n\to \infty} \hat{\tau}^n_{j,k}(\mvy) = y_{j,k}$. Using this along with~\eqref{eqn:bound-for-dct}, we thus obtain
    \begin{align}
        \lim\limits_{n\to\infty} \prod\limits_{j=0}^K \prod\limits_{k=1}^K \left[q^n_j(\hat{T}^n_{j,k}(\mvy), G_n(\mvy))\right] \fJ_n(\mvy) &= \prod_{j=0}^{K}\prod_{k=1}^K h(\sigma_k(\mvy_j))\label{eqn:pointwise-limit-1},\\
        \lim\limits_{n\to \infty}r_j(G_n(\mvy)) &= H(\sigma_k(\mvy_j))\label{eqn:pointwise-limit-2},
    \end{align} where $\mvy_j := (y_{j,k} : 1 \le k \le K)$.
    
    Following the arguments in the proof of Lemma 5.4 of \cite{BBDS04_macro}, it is easy to see that the Lebesgue measure of the set $\set{\mvy : G_n(\mvy) \text{ is \emph{not} } \emph{valid}} \cap [0, \log\left(e^L + K\right)]^{K^2+K}$ converges to zero. We skip the details of this fact. 
    Further, by~\eqref{eqn:pointwise-limit-1},~\eqref{eqn:pointwise-limit-2} and definition of $P_{3,n}(\mvx)$ in~\eqref{eqn:prob-ex-3}, we have on the set $\set{\mvy : G_n(\mvy) \text{ is } \emph{valid}} \cap [0, \log\left(e^L + K\right)]^{K^2+K}$, 
\begin{align*}
 \lim\limits_{n\to \infty}\vP_{3,n}(G_n(\mvy)) \fJ_n(\mvy) =  \prod_{j=0}^K \left[ \exp\left(-H(\sigma_K(\mvy_j))\right)\prod_{k=1}^K h(\sigma_k(\mvy_j))\right].
\end{align*} Hence, by DCT, we have \begin{align*}
    \lim\limits_{n\to \infty}\int_{\bR_+^{K^2+K}} f(\tau^n_{1:K}(\mvx)) \vP_{3,n}(\mvx)d\mvx &=    \lim\limits_{n\to \infty}\int_{\bR_+^{K^2+K}} f(\hat{\tau}^n_{1:K}(\mvy)) \vP_{3,n}(G_n(\mvy)) \fJ_n(\mvy) d\mvy\\
    &=\int_{\bR_+^{K^2+K}} f(\mvy) \prod_{j=0}^K \left[ \exp\left(-H(\sigma_K(\mvy_j))\right)\prod_{k=1}^K h(\sigma_k(\mvy_j))\right] d\mvy\\
    &=\E(f(\sE_{1:K})).
\end{align*} The last line follows from the expression for joint density of $(\sE_{0\leadsto 1},\sE_{1\leadsto 2},\dots,\sE_{k-1\leadsto k})$ in~\eqref{eqn:joint-density}.  This finishes the proof.\end{proof}

\subsection{Proof of local weak convergence in the mesoscopic regime: Theorem \ref{thm:lwc-meso}:}  
\textcolor{black}{The main idea of the proof is to do a} branching process approximation, to the breadth-first exploration of the fringe of a vertex in the process $\cT_n$. Before we prove local weak convergence, we argue that the degree of a vertex can roughly approximated by a Poisson(1) random variable. \textcolor{black}{Let us denote by $d_{TV}$ the total variation distance between two discrete distributions.}
\begin{lemma}[Poisson Approximation]\label{lem:poiss-approx} Let $D_{j,n}$ be the degree of vertex \textcolor{black}{$j$} at time $n$.
    For any $\epsilon > 1$, there exists $j_0$ such that for all $j_0 \leq j \leq n$,  \begin{align*}
        &d_{TV}\left(D_{j,n}, \text{Poisson}(1)\right) \\
        &\leq \epsilon \left[1-(1+\epsilon j^{\beta-1})^{-\beta}\right] + \epsilon j^{-\beta} + \max(\epsilon-1,j^{-\beta})+ \ind\set{N_j\geq n-j} \abs{\frac{n-j}{j^\beta}-1}
    \end{align*}where $N_j = \abs{\set{i\geq j:i-i^\beta \leq j}}.$
\end{lemma}

\textcolor{black}{ To explain the bound above, we call a vertex labeled $j$ a \emph{good vertex} at time $n$ if $N_j \leq n$. Note that $N_j$ is the maximal index of a vertex that can potentially attach to vertex $j$.  In other words, a \emph{good vertex } at time $n$ had seen all the vertices that it could potentially connect to by time $n$. Thus, for good vertices its degree distribution can be approximated well by Poisson distribution. The error made by Poisson approximation for the \emph{good vertices} is accounted by the first three terms, while the last term accounts for the error for \emph{bad vertices}  at time $n$.}
\begin{proof}
For any $j \leq i$, define $Y_{i,j} = \ind\set{\text{vertex }i\text{ connects to  vertex }j}$, we then have $Y_{i,j}\sim \text{Ber}\left(\frac{\ind\set{i-i^\beta \leq j}}{i^\beta}\right).$ Furthermore, for any $j \geq 1$, the collection of random variables $\set{Y_{i,j}:i\geq j}$ are independent. Let $N_j = |\set{i: i-i^\beta \leq j < i}|$ be the total number of vertices that have access to vertex $j$. Also, define $N_{j,n} = \min\set{N_j,(n-j)}$ to be the number of vertices upto time $n$ that have access to vertex $j$. The following lemma gives a bound on $N_j$, \textcolor{black}{and its proof is provided at the end of the section.}
\begin{lem}\label{lem:n_j-bound}
    We have \begin{enumeratei}
        \item $N_j \geq j^\beta-1$ for all $j\geq 1$.
        \item for any $\delta > 1,$ there exists $j_0\geq 1$ such that $N_j \leq \delta j^\beta$ for all $j\geq j_0$.
    \end{enumeratei}
\end{lem}

Using the above lemma and triangle inequality, for sufficiently large $j$, we have \begin{align}\label{eqn:tv-triangle}
    d_{TV}\left(D_{j,n}, Bin\left(N_{j,n},\frac{1}{j^{\beta}}\right)\right) &\leq \sum_{i=j}^{n} \ind\set{i-i^\beta \leq j}\left[\frac{1}{j^\beta}-\frac{1}{i^\beta}\right]\nonumber\\
    &\leq N_{j,n} \left[\frac{1}{j^\beta} - \frac{1}{(j+\epsilon j^\beta)^\beta}\right] \leq \epsilon \left[1-(1+\epsilon j^{\beta-1})^{-\beta}\right].
\end{align} The last two lines follow from (ii) of Lemma \ref{lem:n_j-bound}. Next, we approximate the binomial distribution in \eqref{eqn:tv-triangle} with a Poisson(1) distribution. By triangle inequality, we have\begin{align}\label{eqn:bin-poiss-tv}
&d_{TV}\left(\text{Bin}\left(N_{j,n},\frac{1}{j^{\beta}}\right), \text{Poisson}(1)\right) \nonumber\\
    &\leq d_{TV}\left(\text{Bin}\left(N_{j,n},\frac{1}{j^{\beta}}\right), \text{Poisson}\left(\frac{N_{j,n}}{j^\beta})\right)\right) +  d_{TV}\left( \text{Poisson}\left(\frac{N_{j,n}}{j^\beta}\right), \text{Poisson}(1)\right)\nonumber\\
    &\leq \frac{N_{j,n}}{j^{2\beta}} + \abs{\frac{N_{j,n}}{j^\beta}-1}.
\end{align} In the last line, the first term is bounded is using Poisson approximation of Binomial. The second term is bounded using total variation bound between Poisson distributions. Finally, note that $N_{j,n} \leq \epsilon j^{\beta}$ and $$\abs{\frac{N_{j,n}}{j^\beta}-1} \leq \abs{\frac{N_j}{j^\beta}-1}+ \ind\set{N_j\geq n-j} \abs{\frac{n-j}{j^\beta}-1} \leq \max(\epsilon - 1,j^{-\beta}) + \ind\set{N_j\geq n-j} \abs{\frac{n-j}{j^\beta}-1}$$

Combining the above observations with \eqref{eqn:tv-triangle} and \eqref{eqn:bin-poiss-tv}, we have \begin{align*}
    &d_{TV}\left(D_{j,n}, \text{Poisson}(1)\right)\\
    &\leq \epsilon \left[1-(1+\epsilon j^{\beta-1})^{-\beta}\right] + \epsilon j^{-\beta} + \max(\epsilon-1,j^{-\beta})+ \ind\set{N_j\geq n-j} \abs{\frac{n-j}{j^\beta}-1}.
\end{align*} 
\end{proof}

\begin{proof}[Proof of Theorem \ref{thm:lwc-meso}]
    Fix a tree $\TT \in \bbT$ with $\abs{\TT} = k$. Let $U\sim Unif(0,1)$ be independent of the process $\set{\cT_n:n\geq 1}$. We first show that \begin{equation*}
    \lim_{n\to \infty}\pr\left(F( \lceil n U\rceil, \cT_n) =  \cT\right) = \lim_{n\to \infty} \int_{0}^1 \pr\left(F( \lfloor n u\rfloor, \cT_n) =  \cT\right) du \to \pr\left(\fpm_{\Poi, 1} = \cT \right).
\end{equation*} The proof of local weak convergence follows by appealing to Lemma \ref{thm:aldous-efr-pfr}.

To this extent, fix $u \in [0,1]$. Consider the breadth-first exploration of the fringe of vertex $\lfloor n u\rfloor$, where we stop the exploration if the fringe/subtree rooted at $\lfloor n u\rfloor$ differs from $\cT$ or if the subtree $= \cT$. \textcolor{black}{Note that we only encounter at most $k$ vertices during the exploration, where recall $|\cT|=k$. 
}
\textcolor{black}{We first argue that the $k$ vertices encountered in the exploration are all \emph{ good vertices} at time $n$. This helps in leveraging the Poisson approximation for these $k$ vertices.} Note by letting $\delta = 2$ in (ii) of Lemma \ref{lem:n_j-bound}, there exists $N \in \bN$ such that index of any vertex that attaches to a vertex $j \geq  N$ is bounded at most by $j+2j^\beta$. Therefore, for $n \geq N/u$, we have that index of the vertex that attaches to vertex $\lfloor nu \rfloor$ is bounded by $nu + 2(nu)^\beta$. \textcolor{black}{Repeating this argument, we see that indices of the first $k$ vertices in the fringe is bounded by $nu + 2(nu)^\beta + 2(nu+2(nu)^\beta)^\beta + \dots (k-\text{terms}) = nu + O(n^\beta)$. Therefore, by Lemma \ref{lem:n_j-bound}, we have all $k$ vertices during the exploration are \emph{good vertices} at time $n$.} Thus, by Lemma \ref{lem:poiss-approx}, for any $\epsilon \geq 1$, there exists $\Tilde{N} \geq N/u$ such that for all $1\leq j\leq k$ \begin{align*}
    d_{TV}(D_{l_j,n},\text{Poisson}(1)) &\leq \epsilon \left[1-(1+\epsilon l_j^{\beta-1})^{-\beta}\right] + \epsilon l_j^{-\beta}+ \max(\epsilon-1,l_j^{-\beta}) = \epsilon - 1+ O(1)
\end{align*}for all  $n\geq \Tilde{N}$. Therefore, we have \begin{align*}
    \limsup_{n\to \infty}\left|\pr\left(F( \lfloor n u\rfloor, \cT_n) =  \cT\right) - \pr(\fpm_{\Poi, 1} = \cT)\right| \leq k(\epsilon - 1) \text{ for all }\epsilon >1.
\end{align*} \textcolor{black}{We refer the reader to proof of Lemma 9.6 in \cite{bollobas-irg} for more details about the branching process approximation.} This completes the proof.
\end{proof}

We end this section with the proof of Lemma \ref{lem:n_j-bound}.
\begin{proof}[Proof of Lemma \ref{lem:n_j-bound}]

    (i) For any $j \geq 1$, we have $\lfloor j^\beta\rfloor \leq (j+\lfloor j^\beta\rfloor)^\beta.$ Thus, $(j+\lfloor j^\beta\rfloor) - (j+\lfloor j^\beta\rfloor)^\beta \leq j.$ Therefore, $N_{j,n} \geq \lfloor j^\beta\rfloor \geq j^\beta -1.$

    (ii) If $(j+\delta j^\beta) - (j+\delta j^\beta)^\beta \leq j$ then $\delta \leq  (1+\epsilon j^{1-\beta})^\beta.$ Since, the RHS of the inequality converges to 1, the result follows. 

\end{proof}

\section{Proofs of global results}

\subsection{Analysis of the height in the mesoscopic regime}

In this section, we derive asymptotics for the height of the trees $\set{\cT_n:n\geq 1}$ in the mesoscopic regime. \textcolor{black}{Recall in this regime we have $j(n)=n-n^{\beta}$ for a fixed parameter $\beta \in (0,1)$.} \textcolor{black}{The proof involves the following three step recipe:
\begin{itemize}
    \item[(i)] Deriving asymptotics of the Markov process $\set{L_n(k):k\geq 1}$ which tracks the index of $k^{\text{th}}$ ancestor of vertex $n$.
    \item[(ii)] Using the result in step-(i) to prove asymptotics of the distance $d(1,n)$ between the root $1$ and the vertex $n$.
    \item[(iii)] Arguing that the height of the tree is achieved by a vertex with index in $[(1-\delta) n,n]$ for arbitrary $\delta>0$, with high probability, so that $d(1,n)$ itself is also the height of the tree up to negligible errors. 
\end{itemize}
}

\textbf{Step} (i): For $1\le k\le d(1,n)$, let $L_n(k)$ be the index of the $k^{\text{th}}$ ancestor of vertex $n$. Define $L_n(0)=n$ and $L_n(k)=0$ if $k > d(1,n)$. Note that $\set{L_n(k):k\geq 0}$ is a Markov process on $\set{0,1,2,\dots,n}$ with $0$ being an absorbing state. When $L_n(k) = l > 0$, $L_n(k+1)$ takes values in $\set{l-1,l-2,\dots,l-\lfloor l^\beta \rfloor}$ with equal probability. 

\textcolor{black}{The following proposition proves that the natural LLN limit for the process $L_n(\cdot)$ when scaled appropriately is given by the function $f_{\beta}:\bR_+ \to\bR_+$ defined as 
\begin{align*}
&f_{\beta}(t)= \begin{cases} & \left(1-\frac{1-\beta}{2}t \right)^{1/(1-\beta)},\;\;\text{if}\;\;t\in [0,\frac{2}{1-\beta}];\\&0,\hspace{83 pt}\;\;\text{if}\;\;t>\frac{2}{1-\beta}.    
\end{cases}    
\end{align*} }

\begin{prop}\label{prop:concentration_meso_n}
    Fix $\beta \in (0,1)$, and let $T_\beta = 2/(1-\beta)$. For any $\epsilon>0$ and $T < T_\beta$, there exists constants $C>0$ and $\xi=\xi(\epsilon)>0$ such that
    \begin{align*}
        \prob\left( \sup_{t \leq T}\left|\frac{1}{n}L_n(\lfloor tn^{1-\beta} \rfloor)-f_{\beta}(t)\right|>n^{\frac{-1+\beta}{2}+\epsilon}\right)\leq \exp(-Cn^{\xi}).
    \end{align*}
\end{prop}

\textcolor{black}{Embedding the discrete chain $L_n(k)$ into a continuous-time process $X_n(t)$ allows us to represent the jumps using independent Poisson processes. This makes it easier to prove the concentration inequality using the associated ODE for $f_\beta(t)$, the function $f_\beta(\cdot)$, and properties of Poisson processes. To this extent, consider a continuous time Markov chain denoted by $\set{X_n(t):t\geq 0}$ on state space $\set{0,1,2,\dots,n}$. The waiting time at any state has exponential distribution with rate $(n^{1-\beta})$. From state $l$, the chain moves to any of the states $\set{l-1,l-2,\dots,l-\lfloor l^\beta \rfloor}$ with equal probability. Therefore, the embedded discrete time chain of $\set{X_n(t):t\geq 0}$ has the same law as the process $\set{L_n(k):k\geq 1}.$ Also, the state $0$ is absorbing as noted before.}

We now express the process $X_n(t)$ in terms of time changed Poisson processes. Let $Y_1,Y_2,\dots$ be i.i.d.\ rate $1$ Poisson processes. \textcolor{black}{By Chapter 6 of \cite{kurtz-ethier}, we have \begin{align}\label{eqn:defn-x_n(t)}
    X_n(t) = n - \sum_{i=1}^n i Y_i\left(n^{1-\beta}\int_{0}^t \frac{\ind\set{i\leq \lfloor X_n(s)^\beta\rfloor }}{\lfloor X_n(s)^\beta\rfloor}ds\right).
\end{align}}
    Let $0 = \tau_0 \leq \tau_1 \leq \tau_2 \leq \dots$ be the event times of the process \begin{align}\label{eqn:def-z(t)}
        Z(t) = \sum_{i=1}^n Y_{i}\left(n^{1-\beta}\int_{0}^t \frac{\ind\set{i\leq \lfloor X_n(s)^\beta \rfloor}}{\lfloor X_n(s)^\beta\rfloor}\right).
    \end{align} Let $n_c = \sup_{t}Z(t)$ be the number of transitions in continuous time process $X_n(\cdot)$ before it hits $0$. Similarly, for the discrete time process, define $n_d = \inf\set{k: L_n(k) = 0}$. \textcolor{black}{From the definition of the process $X_n(\cdot)$,  we have $n_d \stackrel{d}{=} n_c <\infty$.} Furthermore the following equality \textcolor{black}{in distribution} holds, \begin{align}\label{eqn:dist-equl-cont-disc}
       \set{X_n(\tau_i):0\leq i\leq n_c} \stackrel{d}{=} \set{L_n(i):0\leq i\leq n_d}. 
    \end{align}
\textcolor{black}{We now have all the ingredients to provide the proof of Proposition \ref{prop:concentration_meso_n}.}

\begin{proof}
   We first prove the concentration for the continuous-time version process $\set{X_n(t):t\geq 0}$, and then leveraging \eqref{eqn:dist-equl-cont-disc} we conclude the result for discrete time process $\set{L_n(k):k\geq 0}.$

   We subtract the compensators of the Poisson processes in \eqref{eqn:defn-x_n(t)}. This gives us a semi-martingale decomposition of the process $X_n(t) = A_n(t) + M_n(t)$, where $A_n(t)$ is a predictable process and $M_n(t)$ is a martingale, both with respect to the natural filtration $\cF_n(t)=\sigma(\set{X_n(s),Y_1(s),\dots,Y_n(s):s\leq t})$, defined as\begin{align*}
        A_n(t) &= n- n^{1-\beta} \int_{0}^t \frac{\sum_{i=1}^n i\ind\set{i\leq \lfloor X_n^\beta(s)\rfloor}}{\lfloor X_n^\beta(s)\rfloor} ds = n-n^{1-\beta}\int_{0}^t \frac{\lfloor X_n^\beta(s)\rfloor + 1}{2} ds\\
        M_n(t) &= -\sum_{i=1}^n iN_i\left(n^{1-\beta}\int_{0}^t \frac{\ind\set{i\leq \lfloor X_n^\beta(s)\rfloor }}{\lfloor X_n^\beta(s)\rfloor}ds\right)
    \end{align*} where $N_i(s) = Y_i(s) - s$ for $1\leq i\leq n$. Let $\bar{X}_n(t) = X_n(t)/n$. Note that for all $t\geq 0$, we have $f_\beta(t) = 1 - \int_{0}^t (f_{\beta}(s))^\beta ds$. Therefore, we have \begin{align*}
        \bar{X}_n(t) - f_{\beta}(t) = -\frac{1}{2}\int_{0}^t \left[\bar{X}_n^\beta(s) - f_{\beta}^\beta(s)\right]ds + \int_{0}^t \frac{1-\{X_n^\beta(s)\}}{2n^\beta}ds + \bar{M}_n(t),
    \end{align*} where $\bar{M}_n(t) = M_n(t)/n$ captures the fluctuations, and $\set{a} = a-\lfloor a \rfloor $ is the fractional part of $a$. Therefore, we have \begin{align}\label{eqn:2929}
        \abs{\bar{X}_n(t) - f_{\beta}(t)} &\leq \int_0^t \abs{\bar{X}_n^\beta(s) - f_{\beta}^\beta(s)} ds + \frac{t}{n^{\beta}} + |\bar{M}_n(t)|.
    \end{align} Furthermore, for all $t \leq T < T_\beta$, we have $f_{\beta}(t) \geq f_{\beta}(T) > 0$. Therefore, there exists a constant $M$ such that $\abs{X_{n}^\beta(s)-f_{\beta}^\beta(s)} \leq M \abs{X_{n}(s)-f_{\beta}(s)}$. Thus, by \eqref{eqn:2929}, we have \begin{align*}
        \abs{\bar{X}_n(t)-f_\beta(t)} \leq M \int_0^t \abs{\bar{X}_n(s) - f_{\beta}(s)} ds + \frac{t}{n^{\beta}} + |\bar{M}_n(t)|.
    \end{align*}
    Therefore, for any $T \geq 0$, by Gronwall's inequality \textcolor{black}{(see Section 5.2 in \cite{oksendall})} we have \begin{align}\label{eq:gronwall-bound}
        \sup_{t\leq T}  \abs{\bar{X}_n(t) - f_{\beta}(t)} \leq \left[\frac{T}{n^{\beta}} + \sup_{t\leq T}|\bar{M}_n(t)|\right] e^{M T}.
    \end{align} Note that $\bar{M}_{n}(\cdot)$ is a continuous time martingale with jumps at most of size $n^{1-\beta}.$ Furthermore, the quadratic variation of the process $M_n(t)$ is given by\begin{align*}
        \langle \bar{M}_n(t)\rangle &= \frac{n^{1-\beta}}{n^2} \int_0^t \frac{\sum_{i=1}^n i^2\ind\set{i\leq \lfloor X_n^\beta(s)\rfloor}}{\lfloor X_n^\beta(s)\rfloor} ds \leq n^{-(1-\beta)} \int_{0}^t \bar{X}_n(s)^{2\beta}ds \leq n^{-(1-\beta)} T
    \end{align*} for all $t\leq T$. By Lemma 2.1 in \cite{vandeGeer} , we have \begin{align*}
        \pr\left(\sup_{t\leq T}|\bar{M}_n(t)| \geq n^{-\frac{(1-\beta)}{2} +\epsilon }\right) &\leq 2 \exp\left(-\frac{n^{-(1-\beta) + 2\epsilon}}{n^{-(1-\beta)/2  + \epsilon} n^{-(1-\beta)} + T n^{-(1-\beta)}} \right)\\
        &\leq \exp\left(-n^{2\epsilon}/T\right).
    \end{align*} Combining the above bound with \eqref{eq:gronwall-bound}, we have \begin{align}\label{eqn:cont-time-bound}
        \pr\left(\sup_{t\leq T} \abs{\bar{X}_n(t) - f_{\beta}(t)} \geq n^{-(1-\beta)/2 +\epsilon}\right) \leq \exp\left(-C n^{2\epsilon}\right)
    \end{align} for some constant $C$ depending on $T$ and $M$. \textcolor{black}{This also proves that we have  $\pr(\tau_{n,c} \leq T_\beta- T+\delta) \leq \exp(-Cn^{2\epsilon}).$ Let $\tilde{Z}_n(T)$ is a Poisson$(n^{1-\beta} T)$ random variable. The last line follows from the observation that $Z_n(\cdot)$ is has jumps according to Poisson $n^{1-\beta}$ process before time $\tau_{n,c}$. Therefore, $Z_n(T) \stackrel{d}{=} \tilde{Z}_n(T)$ when $\tau_{n,c} > T$, and hence $$\pr\left(\abs{Z_n(T)-n^{1-\beta}T}\geq n^{\frac{(1-\beta)}{2}+\epsilon}\right) 
    \leq \pr(\tau_{n,c} \leq T_\beta- T) + \pr\left(\abs{\tilde{Z}_n(T) -n^{1-\beta}T} \geq  n^{\frac{(1-\beta)}{2}+\epsilon}\right) $$ Finally, using large deviation bounds of Poisson random variables, we have $$\pr\left(\abs{Z_n(T)-n^{1-\beta}T}\geq n^{\frac{(1-\beta)}{2}+\epsilon}\right) \leq \exp(-Cn^{2\epsilon}) + \exp\left(-C'n\right)$$ for some constant $C'.$ Using \eqref{eqn:dist-equl-cont-disc} and the above proves the result.} \end{proof}

The following corollary is a \textcolor{black}{straightforward} consequence of \textcolor{black}{Proposition} \ref{prop:concentration_meso_n}. This will be the main tool used in Step(ii).
\begin{cor}
    \label{cor:concentration_meso}
    Fix $\beta \in (0,1)$. For any $\epsilon,\delta>0$ and $T < T_\beta = 2(1-\beta)^{-1}$ there exist constants $C=C(\delta,T)>0$ and $\xi=\xi(\epsilon)>0$ such that
    \begin{align*}
        \prob\left(\exists\;\;v \in [\delta n, n]: \sup_{t\leq T }\left|\frac{1}{v}L_v(\lfloor tv^{1-\beta} \rfloor)-f_{\beta}(t)\right|>v^{\frac{-1+\beta}{2}+\epsilon}\right)\leq \exp(-Cn^{\xi}).
    \end{align*}
\end{cor}
\begin{proof}
\textcolor{black}{
    By a union bound, the LHS above is at most
    \begin{align*}
        \sum_{v = \delta n}^n\prob \left(\sup_{t\leq T }\left|\frac{1}{v}L_v(\lfloor tv^{1-\beta} \rfloor)-f_{\beta}(t)\right|>v^{\frac{-1+\beta}{2}+\epsilon} \right)
    \end{align*}
    which using Proposition \ref{prop:concentration_meso_n} is at most $\sum_{v=\delta n}^n\exp \left(-Cn^{\xi} \right)\leq \exp\left(-Cn^{\xi/2} \right)$.
    }
\end{proof}

\textbf{Step} (ii): We now use the Corollary \ref{cor:concentration_meso} to \textcolor{black}{complete Step (ii), i.e.,} prove convergence after scaling of the distance \textcolor{black}{$d(1,n)$} from $1$ to $n$:
\begin{prop}\label{prop:dist_limit}
For $\beta \in (0,1)$, it holds that 
\[n^{-(1-\beta)} d(1,n)\convp \frac{2}{1-\beta}. \]
\end{prop}
Intuitively, the idea behind the proof is as follows. We start exploring the ancestral path to $1$ starting from $n$. For small $\eps>0$, we explore up to the ancestor with index $\left(\frac{2}{1-\beta}-\eps\right)n^{1-\beta}$ (with the convention that $n$ itself is its ancestor with index $0$, the parent of $n$ is the ancestor with index $1$, \dots). The concentration from Proposition \ref{prop:concentration_meso_n} essentially guarantees that this ancestor itself has label about $\alpha n$, so that now we can consider the distance $d(1,\alpha n)$ in the tree $\cT_{\alpha n}$, i.e., we can work \emph{inside} the tree $\cT_{\alpha n}$. Let us call getting inside the tree $\cT_{\alpha n}$, the \emph{first hop}. Similarly, starting from here, the \emph{second hop} gets us into the tree $\cT_{\alpha^2 n}$, and so on. Repeating these hops $K \log n$ times for a suitably chosen $K$, we can essentially \emph{get inside} the tree $\cT_{\delta n^{1-\beta}}$ for a sufficiently small $\delta>0$, and once we are in this final tree, we can simply bound the distance to $1$ from any vertex in it by the size of this tree, which is $\delta n^{1-\beta}$. Combining this term with the total contribution to $d(1,n)$ from all the hops, we show that the scaled distance $d(1,n)n^{-(1-\beta)}$ is about $\frac{2}{1-\beta}$, perhaps off by a small constant, which suffices to establish Proposition \ref{prop:dist_limit}. Let us now provide the argument.  
\begin{proof}
Fix $\alpha>0$. We will later let $\alpha \to 0$, so think of $\alpha$ being small. Define
\begin{align*}
    t_{\alpha}:=f_{\beta}^{-1}(\alpha),
\end{align*}
and note that $t_{\alpha}<\frac{2}{1-\beta}$ whenever $\alpha>0$, and further, $t_{\alpha} \to \frac{2}{1-\beta}$ as $\alpha \to 0$. Let $v_0=n$, and for $i \geq 0$, inductively define 
    \begin{align*}
        v_{i+1}:=L_{\lfloor t_{\alpha}v_i^{1-\beta} \rfloor}(v_i).
    \end{align*}
For any constant $K>0$, consider the event 
\begin{align*}
    \cS_n(K):=\bigcap_{i=0}^{K \log n}\cE_i,\;\;\text{where}\;\;\cE_i:=\left\{v_{i+1}\in\left[\alpha v_i - v_i^{\frac{1+\beta}{2}+\eps}, \alpha v_i + v_i^{\frac{1+\beta}{2}+\eps}\right] \right\},
\end{align*}
where \textcolor{black}{we choose} $\eps>0$ sufficiently small such that $\frac{1+\beta}{2}+\eps<1$. Let us first check the behavior of $d(1,n)$ on the event $\cS(K)$. Consider the maps $g_L(x)=\alpha x - x^{(1+\beta)/2+\eps}$ and $g_R(x)=\alpha x + x^{(1+\beta)/2+\eps}$. Observe that both $g_L$ and $g_R$ are monotone increasing on any interval of the form $(c,\infty)$, \textcolor{black}{for any $c$ satisfying 
\begin{equation}\label{eq:monotonicity_LB}
c>\left[\left(\frac{1+\beta}{2}+\eps \right)\alpha^{-1} \right]^{\left(\frac{1-\beta}{2}-\eps \right)^{-1}}.
\end{equation}} 
On the event $\cS_n(K)$, for any $0\leq i \leq K\log n$, we have $g_L(v_i)\leq v_{i+1} \leq g_R(v_i)$. Let us denote $M(K)=K \log n$. If we can use that $g_L$ and $g_R$ are monotone increasing, recalling $v_0=n$, inductively, we can obtain 
\begin{equation}\label{eq:final_label}
    g_L^{(i)}(n)\leq v_{i}\leq g_R^{(i)}(n) \ \ \text{ for all } 0\leq i\leq M(K)
\end{equation}
where for any function $f(\cdot)$, $f^{(k)}(\cdot)$ denotes its $k$-fold composition. It is straightforward to see that $g_R(x)$ is monotone increasing on $(0,\infty)$. To argue the monotonicity for $g_L(\cdot)$, it is sufficient to claim that on the event $\cS(K)$, $g_L(v_i)$ \textcolor{black}{is always at least the RHS of \eqref{eq:monotonicity_LB}, which will follow, if we in particular show that $g_L(v_i)$ diverges as $n \to \infty$ for any $1 \leq i \leq M(K)$. For this we need to choose $K$ cleverly, which we make precise next.} 

Note that for $i=0$, $g_L(v_0)=g_L(n)\gg1$. It is enough to check that $g_L^{(M(K))}(n)$ is divergent, where recall $M(K)=K \log n$ for some constant $K>0$. From the definition of $g_{L}(\cdot)$, inductively arguing we note that for all large $n$, $g_L^{(M(K))}(n)\geq (\alpha-\eps_1)^{M(K)}n=n^{1-\log(1/(\alpha-\eps_1))K}\gg 1$ whenever $K<\log(1/(\alpha-\eps_1))^{-1}$, for any $0<\eps_1<\alpha$. Thus, for the monotonicity argument, we require the condition
\begin{equation}\label{eq:monotone_condition}
    M(K)<\log(1/(\alpha-\eps_1))^{-1} \log n\;\;\text{for some}\;\;0<\eps_1<\alpha.
\end{equation}


Decomposing the distance
\begin{equation}\label{eq:dist_decomp}
    d(1,n)=\sum_{i=1}^{M(K)+1} d(v_{i-1},v_i)+d(v_{M(K)+1},1),
\end{equation}
we thus observe that on the event $\cS(K)$, we have the upper bound
\begin{align*}
    d(1,n)\leq \sum_{i=1}^{M(K)+1} (t_{\alpha}v_{i-1}^{1-\beta}+1)+d(v_{M(K)+1},1).
\end{align*}
Next, we note that by \eqref{eq:final_label}, the vertex $v_{M(K)+1}$ is contained in the tree $T_{g_R^{(M(K))}(n)}$, which means we can simply bound the distance $d(v_{M(K)+1},1)\leq g_R^{(M(K))}(n)$. Furthermore, using the upper bounds $v_{i-1}\leq g_R(v_i)$ and the monotonicity of the function $g_{R}$, we thus obtain
\begin{align*}
    d(1,n)\leq t_{\alpha}\sum_{i=0}^{M(K)}(g_R^{(i)}(n))^{1-\beta}+M(K)+g_R^{(M(K))}(n),
\end{align*}
where we use the convention $g_R^{(0)}(n)=n$. By some easy analysis, it is easy to check that $g_R^{(i)}(n)\leq (\alpha+\eps')^{i}n$ for any $\eps'>0$ and $i=0,1,\dots$, so that we obtain
\begin{align*}
    d(1,n)\leq \frac{t_{\alpha}n^{1-\beta}}{1-(\alpha+\eps')^{1-\beta}}+M(K)+(\alpha+\eps')^{M(K)}n.
\end{align*}
Let 
\begin{align*}
    M(K)=\left\lfloor \frac{\beta \log n}{\log(1/(\alpha+\eps'))}+\frac{\log(1/\eps'')}{\log(1/(\alpha+\eps))}\right\rfloor,
\end{align*}
for any $\eps''>0$, and let us immediately observe that by choosing $\beta,\eps'>0$ sufficiently small (as a function of $\eps_1$) the condition \eqref{eq:monotone_condition} is satisfied. We thus obtain
\begin{align*}
    d(1,n)\leq \left(\frac{t_{\alpha}}{1-(\alpha+\eps')^{1-\beta}}+\eps'' \right)n^{1-\beta}+o(n^{1-\beta}).
\end{align*}
Thus on the event $\cS_n(K)$, \textcolor{black}{by choosing all of} $\alpha,\eps,\eps', \eps_1$ \textcolor{black}{sufficiently small,} 
and recalling $t_{\alpha}\to \frac{2}{1-\beta}$ as $\alpha\to 0$, we obtain that for all large $n$,
\begin{align*}
    d(1,n)\leq \left(\frac{2}{1-\beta}+\delta \right)n^{1-\beta}
\end{align*}
for arbitrary $\delta>0$. A similar argument, starting again from \eqref{eq:dist_decomp} and now working with $g_L$ instead of $g_R$, establishes a lower bound on $d(1,n)$ on the event $\cS_n(K)$ as 
\begin{align*}
    d(1,n)\geq \left(\frac{2}{1-\beta}-\delta \right)n^{1-\beta},
\end{align*}
so that we need only check $\cS_n(K)$ holds with high probability. To this end, we write
\begin{align*}
    \prob\left(\cS_n(K)^c \right) \leq \sum_{i=0}^{M(K)} \prob(\cE_i^c).
\end{align*}
Note that,
\begin{align*}
    \prob(\cE_i^c)\leq \prob(\cE_i^c \cap \cE_{i-1})+\prob(\cE_{i-1}^c).
\end{align*}
Furthermore, conditionally on $v_i$, applying Corollary \ref{cor:concentration_meso} on the tree $T_{v_i}$, we obtain for any $0<\delta_1<\alpha$, and $i \geq 1$,
\begin{align*}
    \prob(\cE_i^c)\leq \E\left[\ind_{\cE_{i-1}}\exp(-Cv_i^{\xi(\eps)}) \right]+\prob(\cE_{i-1}^c)& \leq \exp(-Cg_L^{(i-1)}(n)^{\xi(\eps)})+\prob(\cE_{i-1}^c)\\& \leq \exp(-C((\alpha-\delta_2)^in)^{\xi(\eps)})+\prob(\cE_{i-1}^c), 
\end{align*}
{for $C=C(\alpha-\delta_1,t_{\alpha})$, a constant $\xi(\eps)>0$, where recall the constants $C=C(\delta,T)$ and $\xi(\eps)$ appearing in the statement of Corollary \ref{cor:concentration_meso}, and for} any $\delta_2>0$, where we again use the fact that $g_L^{(i-1)}(n)\geq (\alpha-\delta_2)^i n$ for any $i \geq 1$.  Applying recursively, we thus obtain
\begin{equation}\label{eq:error_bound_meso}
    \prob(\cS_n(K)^c)\leq \sum_{i=1}^{M(K)}(M(K)-i)\exp(-C((\alpha-\delta_2)^in)^{\xi(\eps)}).
\end{equation}
whenever $M(K)=O(\log n)$. Fixing $\delta_1=\delta_2=\alpha/2$ and letting $n \to \infty$, we note that the RHS of the last display is $o(1)$. This finishes the proof.
\end{proof}

\begin{rem}\label{rem:dist_tail_bound}
    Observe that analyzing the error bound \eqref{eq:error_bound_meso}, we in fact obtain the following tail bound for any $\eps>0$,
    \begin{equation}\label{eq:dist_tail_bound}
        \prob\left(\left|n^{-(1-\beta)}d(1,n)-\frac{2}{1-\beta}\right|\geq \eps\right)\leq \prob(\cS_n(K)^c)\leq \exp\left(-K(\eps)n^{\zeta} \right),
    \end{equation}
    for some $K(\eps)>0$ and $\zeta>0$ for all large $n$.
\end{rem}
\textcolor{black}{With this remark at hand, we can finally complete \textbf{Step} (iii). Note that the next proposition proves Theorem \ref{thm:ht_meso}.}

\begin{prop}
    Fix $\beta \in (0,1)$ and let $H_n(\beta)$ be the height of the tree $T_n$. Then
    \begin{align*}
        \frac{H_n(\beta)}{n^{1-\beta}}\convp \frac{2}{1-\beta}.
    \end{align*}
\end{prop}
\begin{proof}
Fix $\eps>0$. Note that by Proposition \ref{prop:dist_limit}, since the height $H_n$ is always larger than $d(1,n)$, we have
\begin{align*}
    \prob\left(H_n>(1-\eps)\frac{2}{1-\beta}n^{1-\beta} \right)\to 1,
\end{align*}
as $n \to \infty$, so that we need only check
\begin{equation}\label{eq:height_UB}
    \prob\left(H_n<(1+\eps)\frac{2}{1-\beta}n^{1-\beta} \right)\to 1.
\end{equation}
We make use of Remark \ref{rem:dist_tail_bound} in this proof. Fix $\alpha_1,\alpha_2>0$ small. Note that by the tail bound \eqref{eq:dist_tail_bound} and a union bound,
\begin{equation}\label{eq:UB_ht_not_by_young_vertices}
\begin{split}
    &\prob\left(\exists\;\;v\in[\alpha_1n^{1-\beta},(1-\alpha_2)n]:d(v,1)> \frac{2}{1-\beta}((1-\alpha_2)^{1-\beta}+\alpha_3)n^{1-\beta}\right)\\& \leq \sum_{v=\alpha_1 n^{1-\beta}}^{(1-\alpha_2) n}\prob\left(d(v,1)>\left(\frac{2}{1-\beta}+\alpha_3\right)v^{1-\beta} \right)\leq \sum_{v=\alpha_1n^{1-\beta}}^{(1-\alpha_2)n}\exp\left(K(\alpha_3)v^{\zeta} \right)=o(1)
    \end{split}
\end{equation}
\textcolor{black}{as $n \to \infty$,} for some small $\alpha_3>0$. Thus with high probability, none of the vertices in $[\alpha_1n^{1-\beta},(1-\alpha_2)n]$ can have height more than $\frac{2}{1-\beta}((1-\alpha_2)^{1-\beta}+\alpha_3)n^{1-\beta}$. The height of any vertex $v \in [1,\alpha_1n^{1-\beta}]$ is trivially at most $\alpha_1n^{1-\beta}$. Since the vertex $n$ has height larger than $\frac{2}{1-\beta}((1-\alpha_2)^{1-\beta}+\alpha_3)n^{1-\beta}$ with high probability, we note that the height of the tree is achieved by a vertex $v \in [(1-\alpha_2)n,n]$ with high probability. But again by a union bound in the spirit of \eqref{eq:UB_ht_not_by_young_vertices} \textcolor{black}{combined with} the tail bound \eqref{eq:dist_tail_bound}, we can argue that with high probability, none of the vertices $v \in [(1-\alpha_2)n,n]$ can have height lying outside the interval $[\frac{2}{1-\beta}v^{1-\beta}-\alpha_4n^{1-\beta},\frac{2}{1-\beta}v^{1-\beta}+\alpha_4n^{1-\beta}]$, for $\alpha_4>0$ arbitrarily small. By choosing $\alpha_2$ and $\alpha_4$ sufficiently small (as a function of $\eps$), thus the height of the tree $H_n$ is at most $(1+\eps)\frac{2}{1-\beta}n^{1-\beta}$ with high probability, which proves \eqref{eq:height_UB}. \end{proof}

\subsection{A global phase transition at $\beta=1/2$ in the mesoscopic regime}
{

\begin{proof}[Proof of Theorem \ref{thm:ghp_meso}]
Let $\beta<1/2$. In Proposition \ref{prop:dist_limit} we showed that $n^{-(1-\beta)}d_{\mathcal{T}_n}(1,n)\convp \frac{2}{1-\beta}$ as $n\to\infty$.
Now, fix $\varepsilon>0$ small enough so that $\beta+\varepsilon<1-\beta$. We will show that with high probability, for each $m<n$, the distance from vertex $m$ to the ancestral line of $n$ satisfies $d(m,m\wedge n)<n^{\beta+\varepsilon}$, which shows that, with high probability, the path from the root to vertex $n$ is at Hausdorff distance $o(n^{1-\beta})$ from $\cT_n$ and thus proves the theorem. 

Let $G_n\sim \operatorname{Geom}(n^{-\beta})$ a Geometric random variable with for $k\in \N$, $\pr(G_n>k)=(1-n^{-\beta})^k$. We will show that for each $m<n$ \begin{equation}\label{eq:geom_stochdom}d(m,m\wedge n)\prec_{st}G_n+1.\end{equation} 
Then, 
\[\pr\left(\exists m<n: d(m,m\wedge n)>n^{\beta+\varepsilon} \right) \le n (1-n^{-\beta})^{n^{\beta+\epsilon}}\le n\exp(-n^\epsilon)\to 0\]
as $n\to \infty$. Thus we focus on showing \eqref{eq:geom_stochdom}. 

We condition on the unique path
\[P_n=(1=L_n(d(1,n))<L_n(d(1,n)-1)<\dots< L_n(0)=n)\] in $\mathcal{T}_n$ from vertex $1$ to vertex $n$. We will show \eqref{eq:geom_stochdom} by iteratively revealing the ancestors of vertex $m$ and showing that, at each step, the next ancestor is on $P_n$ with probability at least $n^{-\beta}$. Formally, we will bound
\begin{equation}\label{eq:prob_ancestor} \pr(L_m(k+1)\in P_n | P_n, L_m(1), \dots, L_m(k))\end{equation}
from below by $n^{-\beta}$.

We use that for any $2\le m\le n$, $[j(m-1), m-1]\cap P_n$ is non-empty. Indeed, let $\ell$ be maximal such that $L_n(\ell)\ge m$. If $L_n(\ell)=m$, then the claim follows immediately. Otherwise, \[L_n(\ell+1)\in [j(L_n(\ell)-1) ,  m-1]\subset [j(m-1), m-1]  \] because $j(n)$ is non-decreasing. 

Then, to bound \eqref{eq:prob_ancestor} we see that either $L_m(k)\in P_n$, in which case \eqref{eq:prob_ancestor} equals $1$. Otherwise, on the conditioning, $L_m(k+1)$ is a unique element from $[j(L_m(k)-1), L_m(k)-1]$. By our observation, at least one element of this set is also in $P_n$. Moreover, the set has size at most $n^{\beta}$, so \eqref{eq:prob_ancestor} is at least $n^{-\beta}$. This implies \eqref{eq:geom_stochdom}. 

\end{proof}

To prove Theorems~\ref{thm:meso_star} and~\ref{thm:meso_not_crt}, we first introduce an algorithm that for $k\ge 1$ and $n=n_1>n_2>\dots> n_k\ge j(n-1)$ allows us to explore the ancestral lines of $n_1,\dots, n_k$ up to the highest branchpoint in the subtree of $\cT_n$ spanned by $1,n_1,\dots, n_k$. Moreover, applying the argument inductively allows us to reveal the full subtree of  $\cT_n$ spanned by $1,n_1,\dots, n_k$. 

The algorithm explores the ancestral lines of $n_1,\dots, n_k$ by iteratively revealing the ancestor of the last explored vertex across the $k$ paths that has the largest label. The algorithm terminates when the highest branchpoint is reached. For $i=1,\dots, k$ we define a sequence of indices $M^i_0,M^i_1,\dots$ so that at time $m$, we have explored $M^i_m$ ancestors of $n_i$. We let $L(m)$ denote the label of $m$th vertex to get revealed.

Set $m=0$ and $M^1_0=\dots=M^k_0=0$ and do the following.
\begin{description}
    \item[(1)] If there is $1\le \ell<\ell'\le k$ for which $L_{n_\ell}(M^\ell_{m})=L_{n_{\ell'}}(M^{\ell'}_m)$ then set $T=m$ and terminate.  
    \item[(2)] Otherwise, set $\ell=\argmax_{1\le i\le k}L_{n_i}(M^i_m)$.  We explore the next ancestor of $n_\ell$ with label $L_{n_\ell}(M^\ell_m+1)$ so we set $L(m+1)=L_{n_\ell}(M^\ell_m+1)$. For $1\le i \le k$, set $M^i_{m+1}=M^i_{m}+\ind_{i=\ell}$, increase $m$ by $1$ and return to {\bf (1)}.
   \end{description}
   
Observe that for all $m$, for $\ell=\argmax_{1\le i \le k} L_{n_i}(M^i_m)$, for any $i$, it holds that $L_{n_i}(M^i_m)\ge j(L_{n_\ell}(M^\ell_m)-1)$. Indeed, if $M^i_m=0$ this follows from the assumption that $n_i\ge j(n)$, because $L_{n_\ell}\le n$ and $j(n)$ is non-decreasing.  Otherwise, by construction, $L_{n_i}(M^i_m-1)>L_{n_\ell}(M^\ell_m)$ so since $L_{n_1}(M^i_m)\ge j(L_{n_i}(M^i_m-1)-1)$ and $j(n)$ is non-decreasing, it follows that $L_{n_1}(M^i_m)\ge j(L_{n_\ell}(M^\ell_m)-1)$.

Before proving the theorems, we will first show that for all $1\le i \le k$, $M^i_m$ is uniformly close to $m/k$ as long as we do not yet explore vertices with very low labels. We also show that $L(m)$ has a limit under rescaling. Recall the definition of $f_\beta$ from Proposition~\ref{prop:concentration_meso_n}. 

\begin{lem}\label{lem:explore_ancestors}
 For $k\ge 1$ and $n=n_1>n_2>\dots> n_k\ge j(n-1),$ for any $\delta>0$, for $\tau=\min\{m:L(m)\le \delta n \}$, for $i=1,\dots, k$,
\[n^{-(1-\beta)}\sup_{0\le m \le T\wedge\tau} |M^i_m-m/k|\convp 0\quad \text{ and }\quad \sup_{0\le m \le T\wedge\tau}\left|n^{-1}L(m)-f_\beta(n^{-(1-\beta)}m/k)\right|\convp 0 \]
\end{lem}
\begin{proof}
Observe that 
\[|M^i_m-m/k| = \frac{1}{k}\left |kM^i_m-\sum_{j=1}^k M^j_m\right |\le \frac{1}{k}\sum_{j=1}^k |M_m^i-M_m^j|\]
because $m=\sum_{j=1}^k M^j_m$, so it is enough to show that \[\sup_{0\le m \le T\wedge\tau} |M_m^i-M_m^j|\convp 0\] for all $1\le i,j\le k$. We see that  $|M_m^i-M_m^j|$ is the difference between the number of ancestors of $n_i$ and $n_j$ explored by time $m$. Since we explore the union of their ancestors by decreasing label, this equals the difference between the number of ancestors of $n_i$ and $n_j$ with label at least $a$ for some $a$. Since we only consider $m\le \tau$, we only consider $a> \delta n$. Thus,
\begin{align*}
\sup_{0\le m \le T\wedge\tau} |M_m^i-M_m^j|&\le \sup_{\delta n < a \le n}\left|\#\left\{a \le v \le n: v\preceq n_i\right\}-\#\left\{a \le v \le n: v\preceq n_j\right\}\right|\\
&= \sup_{\delta n<a \le n}\left|\sup\left\{v:L_{n_i}(v)\ge a\right\}-\sup\left\{k:L_{n_j}(k)\ge a\right\}\right|.
\end{align*}
 Proposition~\ref{prop:concentration_meso_n}, applied with $n$ replaced by $n_i$ or $n_j$, implies that this upper bound converges to $0$ in probability when divided by $n^{1-\beta}$. Indeed, this proposition implies that $\frac{1}{n}L_{n_i}(\lfloor t n^{1-\beta}\rfloor)\ind\set{L_{n_i}(\lfloor t n^{1-\beta}\rfloor)>\delta n/2 }$ and $\frac{1}{n}L_{n_j}(\lfloor t n^{1-\beta}\rfloor)\ind\set{L_{n_j}(\lfloor t n^{1-\beta}\rfloor)>\delta n/2 }$ converge uniformly in probability to the same non-increasing function, which implies uniform convergence of their hitting times as well. 
 
The second statement follows from Proposition~\ref{prop:concentration_meso_n},
\begin{align*}\left|n^{-1}L(m)-f_{\beta}(n^{-(1-\beta)}m/k)\right| \le  \sum_{i=1}^k\left|n^{-1}L_{n_i}(M^i_m)-f_\beta(n^{-(1-\beta)}m/k)\right| ,\end{align*}
 and the fact that $M^1_m,\dots, M^k_m$ are uniformly close to $m/k$ by the first statement. 
\end{proof}

\begin{proof}[Proof of Theorem \ref{thm:meso_star}]
We fix $k\ge 1$ and $\varepsilon>0$. We use the algorithm to explore the ancestral lines of $n_1=n, n_2=n-1,\dots, n_k=n-k+1$. Observe that for $T$ the time at which the algorithm terminates, $T$ is a lower bound for the number of vertices on the union of the ancestral lines of $n_1,\dots, n_k$. Thus, by Proposition~\ref{prop:dist_limit}, it suffices to show that for  $n$ large enough, with high probability,
\[T\ge \left(  \frac{2k}{1-\beta}-\varepsilon \right)n^{1-\beta} .\] 

Fix $\delta>0$ and let  $\tau=\min\{k:L(k)\le \delta n \}$.
We will show that with high probability, \[T\wedge \tau \ge  \left(  \frac{2k}{1-\beta}-\varepsilon\right) n^{1-\beta},\] which proves the statement. 

Observe that conditional on $L_{n_i}(1), \dots, L_{n_i}(M^i_{m})$ for $1\le i \le m$, if $T\wedge \tau >m$ then the probability that $T=m+1$ is at most $(k-1)(\delta n)^{-1}$. Indeed, if $T>m$ then for $\ell=\argmax_{1\le i\le k}L_{n_i}(M^i_m)$, we explore the next ancestor of $n_\ell$. Under the conditioning, this ancestor is a uniform pick from the interval 
$[j(L(m)-1), L(m)-1]$ that contains $\{L_{n_i}(M^i_m):i\ne\ell\}$. Since $\tau>m$ this interval has length at least $(\delta n)^\beta$, so the ancestor equals one of $\{L_{n_i}(M^i_m):i\ne\ell\}$ with probability at most $(k-1)(\delta n)^{-\beta}$. Thus, for any $K\ge 1$, 
\begin{align} \begin{split}\label{eq:bound_T}\pr(T\wedge \tau \le  K)
&\le \pr(L(K)\le \delta n)+\pr(\text{Bin}(K, (k-1) (\delta n )^{-\beta})>0).\end{split}
\end{align}

We see that, by Lemma~\ref{lem:explore_ancestors} for \[K=\left \lfloor \left(  \frac{2k}{1-\beta}-\varepsilon \right) n^{1-\beta}\right\rfloor \quad \text{ and }\quad \delta= \frac{1}{2}f_\beta\left(\frac{2}{1-\beta}-\frac{\varepsilon}{k} \right) \] the first term in \eqref{eq:bound_T} goes to $0$. By Markov's inequality, since $\beta>1/2$, the second term goes to $0$ as well. This implies the claim, and the theorem follows. 

\end{proof}
For the proof of Theorem~\ref{thm:meso_not_crt}, we first prove the following lemma, that only studies the height of the highest branchpoint of a subtree of $\cT_n$ spanned by $k$ vertices. 
\begin{lem}\label{lem:1st_step_ancestral}
Fix $k\ge 1$ and let $n=n_1>n_2>\dots> n_k\ge j(n-1)$. Then, for $\beta=1/2$ and $D_n$ the distance from $1$ to the highest branchpoint in the subtree of $\cT_n$ spanned by $1,n_1,n_2,\dots, n_k$, it holds that $n^{-1/2}D_n\convd 4X_{k-1}$ where $X_{k-1}$ is distributed as $U^{1/(4k(k-1))}$ for $U$ a $\operatorname{uniform}[0,1]$ random variable. Moreover, jointly, $n^{-1}L(T)\convd X_{k-1}^2$. \end{lem}
\begin{proof}
We conduct a comparable analysis as in the proof of Theorem~\ref{thm:meso_star}. Fix $\delta>0$ and let  $\tau=\min\{k:L(k)\le \delta n \}$. Observe that conditional on $L_{n_i}(1), \dots, L_{n_i}(M^i_{m})$ for $1\le i \le m$, if $T\wedge \tau >m$ then for $\ell=\argmax_{1\le i\le k}L_{n_i}(M^i_m)$, we explore the next ancestor of $n_\ell$. Under the conditioning, this ancestor is a uniform pick from the interval 
$[j(L(m)-1), L(m)-1]$ that contains $\{L_{n_i}(M^i_m):i\ne\ell\}$. This interval has length $\lfloor (L(m)-1)^{1/2} \rfloor$, so the ancestor equals one of $\{L_{n_i}(M^i_m):i\ne\ell\}$ with probability $(k-1)\lfloor (L(m)-1)^{1/2} \rfloor^{-1}$. Thus, 
\begin{align*}&\pr(T> m+1|
L_{n_1}(1), \dots, L_{n_1}(M^1_{m}), \dots, L_{n_k}(1), \dots, L_{n_k}(M^k_{m}),T\wedge \tau > m)\\&\quad = 1-(k-1)\lfloor (L(m)-1)^{1/2} \rfloor^{-1} .\end{align*}
Thus, inductively, we see that, for any $\varepsilon>0$, 
\begin{align*}\pr(T> m+1)&\le \prod_{i=1}^m \left(1-(k-1)n^{-1/2} (f_{1/2}(n^{-1/2} i/k)+\eps)^{-1/2}\right)\\
&\quad +\pr\left (\sup_{0\le i \le T\wedge\tau}\left|n^{-1/2}\lfloor (L(i)-1)^{1/2} \rfloor-f_{1/2}(n^{-1/2}i/k)^{1/2}\right|>\eps\right) \\
&\quad +\pr(L(m)\le \delta n)\end{align*}
and 
\begin{align*}\pr(T> m+1)&\ge \prod_{i=1}^m(1-(k-1)n^{-1/2} (f(n^{-1/2} i/k)-\eps)^{-1/2} )\\
&\quad -\pr\left (\sup_{0\le i \le T\wedge\tau}\left|n^{-1/2}\lfloor (L(i)-1)^{1/2} \rfloor-f_{1/2}(n^{-1/2}i/k)^{1/2}\right|>\eps\right) \\
&\quad -\pr(L(m)\le \delta n).\end{align*}
By Lemma~\ref{lem:explore_ancestors} and Proposition \ref{prop:concentration_meso_n}, for $0<t<1$, using that $T_{1/2}=4$, for   $\delta,\varepsilon<f_{1/2}(4t)$
\begin{align*}&\prod_{i=1}^{\lfloor 4k t n^{1/2}\rfloor}(1-(k-1)n^{-1/2} (f_{1/2}(n^{-1/2} i/k)-\eps)^{-1/2} )+o(1)\\&\quad \le  \pr(T\ge \lfloor 4kt n^{1/2}\rfloor)\le\prod_{i=1}^{\lfloor 4k t n^{1/2}\rfloor}(1-(k-1)n^{-1/2} (f_{1/2}(n^{-1/2} i/k)+\eps)^{-1/2} )+o(1) .\end{align*}
By a standard argument, writing the product as an exponential of a sum of logarithms, considering the Taylor expansion of $\log(1+x)$ for $x$ near $0$ and recognizing a Riemann sum in the exponent, the left-hand side and right-hand side converge to 
$\exp(-(k-1)\int_{0}^{4kt}(f_{1/2}(s/k)-\varepsilon)^{-1/2} ds) $ and $\exp(-(k-1)\int_{0}^{4kt}(f_{1/2}(s/k)+\varepsilon)^{-1/2} ds)$ respectively. Thus, since $\varepsilon$ was arbitrary, changing variables in the integral, we see that \[\pr(T\ge \lfloor 4kt n^{1/2}\rfloor)\to \exp\left(-4k(k-1)\int_{0}^tf_{1/2}(4s)^{-1/2} ds\right) \] 
as $n\to \infty$. Recalling that 
$f_{1/2}(t)=(1-t/4)^2$ for $t\in[0,4]$, we see that the limit equals $(1-s)^{4k(k-1)}$. Thus, since $n^{-1/2}|M^1_T-T/k|\convp 0$ by Lemma~\ref{lem:explore_ancestors}, it holds that the difference in height between $n$ and the highest branchpoint is $n^{1/2}4(1-U^{1/(4k(k-1))})(1+o_p(1))$. Moreover, by Proposition~\ref{prop:dist_limit}, $|n^{-1/2}d(1,n)-4|\convp 0$. Combining these two facts then implies that the distance from the root to the highest branchpoint is $n^{1/2}4U^{1/(4k(k-1))}(1+o_p(1))$ as claimed.  The convergence of $L(T)$ under rescaling  follows from Lemma~\ref{lem:explore_ancestors}. \qedhere \end{proof}

\begin{proof}[Proof of Theorem~\ref{thm:meso_not_crt}]
We first apply Lemma~\ref{lem:1st_step_ancestral} to $n_1=n, n_2=n-1,\dots, n_k=n-k+1$ to find the claimed asymptotic height of the highest branchpoint. Moreover, we observed that for $\ell=\argmax_{1\le i \le k} L_{n_i}(M^i_T)$, for any $i$, it holds that $L_{n_i}(M^i_T)\ge j(L_{n_\ell}(M^\ell_T)-1)$. Thus, we may apply Lemma~\ref{lem:1st_step_ancestral} for $k-1$, with $n_1,\dots, n_{k-1}$ the distinct values in $\{L_{n_i}(M^i_T)\}$ to find that the distance to the root of the next branchpoint is $(L(T))^{1/2}4U^{1/(4(k-1)(k-2))}(1+o_p(1))$, so, by the convergence in distribution of $L(T)$ in Lemma~\ref{lem:1st_step_ancestral}, we deduce the joint asymptotic law of the highest and second highest branchpoint. The result then follows inductively. \end{proof}

}

\subsection{Analysis of the height in the macroscopic regime}
We start with the proof of Prop. \ref{prop:bp-height} following the relatively well trodden path pioneered in \cite{pittel1994note} using \cite{kingman1975first}. By Lemma \ref{lem:malthus} we have that on the set of survival $\set{|\BP_{\theta}(\infty)| = \infty}$, 
\begin{equation}
    \label{eqn:stop-as}
    T_n - \log{n} \convas -\log{W_\theta}. 
\end{equation}
Next, for each fixed $k\geq 1$ define the stopping time $B_k$ as the first time there exists an individual in generation $k$ in $\BP_\theta$. Then \cite{kingman1975first} shows that on the event of survival, 
\begin{equation}
    \label{eqn:first-birth-as}
    B_k/k\convas \kappa(\theta), \qquad \text{as } k\to\infty,
\end{equation}
where $\kappa(\theta)$ is as in Prop. \ref{prop:bp-height}, defined in \eqref{eqn:def-ht-con}. Now by the definition of the height, 
\[B_{\cH(\cT_{n,\theta}^{\BP})} \leq T_n \leq B_{\cH(\cT_{n,\theta}^{\BP})+1} \]
Combining \eqref{eqn:stop-as} and \eqref{eqn:first-birth-as} completes the proof of the proposition. 

\qed

Now \textcolor{black}{we use} the results of \cite{devroye2012depth} which analyzed \textcolor{black}{the model of scaled attachment random recursive trees (SARRT). Recall from the definition of this model from Section \ref{sec:SARRTs} that in this model, vertex $n$ connects to a vertex $\lfloor n X_n \rfloor$ where $\{X_i:i\geq 0\}$ is i.i.d.\ from a given distribution $\mu$. } 

In the macroscopic regime, the model considered in this paper is a special case of SAART model with $\mu = \text{U}[\theta , 1]$. Now let $X\sim \text{U}[\theta , 1]$ and let $\mu = \E( - \log(X))$. Define 
\begin{equation}
    \label{eqn:DFW-def}
    \Lambda(\lambda) = \log(\E(X^\lambda)), \qquad \Lambda^*(z) = \sup_{\lambda \in \bR}(\lambda z - \Lambda(\lambda))
\end{equation}
Next let $\Psi(c) = c\Lambda^*(-1/c)$ and define 
\begin{equation}
    \label{eqn:alpha-max}
    \alpha_{\max} = \inf\set{c: c> 1/\mu, \text{ and } \Psi(c) > 1}.
\end{equation}
It is easy to check that $\alpha_{\max} = [\kappa(\theta)]^{-1} $. Then \cite{devroye2012depth}*{Theorem 2} shows that $H_n/\log{n} \probc \alpha_{\max}$. This completes the proof of Theorem \ref{thm:ht-macro}. \qed 

\section{Local limits of SARRTs}
\begin{proof}[Proof sketch of Theorem \ref{thm:lwc-saart}.]
    \textcolor{black}{The proof is identical to that of Theorem \ref{thm:lwc-macro} except for the following changes. \begin{enumerate}
        \item The quantity $p_{j}^n(i,x)$ defined in \eqref{eqn:prob-step} in this case would be \begin{align*}
            p_{j}^n(i,x) &= \ind\set{i\leq T_{j,K}^n(\mvx)} \int_{\rho_{j}^n(\mvx)/i}^{(\rho_{j}^n(\mvx)+1)/i} f_{V}(u)du\\
            &= \ind\set{i\leq T_{j,K}^n(\mvx)} \E(\ind\set{\rho_{j}^n(\mvx)/V \leq i < (\rho_{j}^n(\mvx)+1)/V})
        \end{align*}
        \item Using approximations as in step(ii) of the proof of Theorem \ref{thm:lwc-macro}, we have \begin{align*}
            \prod_{i\in \cI_{\mvx}}(1-\sum_{j=0}^K p_{j}^n(i,\mvx)) &\approx \exp\left(\sum_{i\in \cI_{\mvx}}\log\left(1-\sum_{j=0}^K p_{j}^n(i,\mvx)\right)\right)\\
            &\approx \exp\left(-\sum_{j=0}
        ^K \sum_{i\in \cI_{\mvx}}\ind\set{i\leq T_{j,K}^n(\mvx)} \E(\ind\set{\rho_{j}^n(\mvx)/V \leq i < (\rho_{j}^n(\mvx)+1)/V})\right)\\
        &\approx \sum_{j=0}^K \E\left(\frac{1}{V} \ind\set{-\log V \leq \log(T_{j,K}^n(\mvx)/\rho_{j}^n(\mvx)) }\right)
        \end{align*}
    \end{enumerate}  The remainder of the proof follows as in Theorem \ref{thm:lwc-macro}.}
\end{proof}

\section*{Acknowledgements} 
This project was initiated during the 2024 BIRS-CMI meeting “Mathematical Foundations of Network Models and their Applications” (24w4004), supported by BIRS, CMI and NHBM. Part of the work was carried out while OA, SB and SD were in residence at the Simons Laufer Mathematical Sciences Institute in Berkeley, California, during the Spring 2025 semester. They gratefully acknowledge the financial support of the National Science Foundation under Grant No. DMS-1928930. SD acknowledges the financial support of the CogniGron research center and the Ubbo Emmius Funds (University of Groningen). Her research was also partially supported by the Marie Skłodowska-Curie grant GraPhTra (Universality in phase transitions in random graphs), grant agreement ID 101211705. SB and AS were partially supported by NSF DMS-2113662, DMS-2413928, and DMS-2434559.  SB was also partially funded by NSF RTG grant DMS-2134107. NM was supported by an AMS Simons travel grant for a visit to the Department of Statistics and Operations Research at the University of North Carolina, Chapel Hill, where part of this work was carried out. All the authors would like to thank Louigi Addario-Berry for many invaluable discussions.

\bibliographystyle{abbrvnat}
\bibliography{ref}
\end{document}